\documentclass{article}
\usepackage[utf8]{inputenc}
\usepackage[a4paper,headheight=13.5pt,width=160mm,top=25mm,bottom=25mm,bindingoffset=6mm]{geometry}
\usepackage{hyperref}
\usepackage{amsfonts}
\usepackage{amsmath}
\usepackage{amsthm}
\usepackage{mathtools}
\usepackage{dsfont}
\usepackage{etoolbox}
\usepackage{bigints}
\usepackage{color}
\usepackage{stmaryrd}
\usepackage{comment}
\newcommand{\R}{\mathbb{R}}

\newcommand{\beq}{\begin{equation}}
\newcommand{\eeq}{\end{equation}}
\newcommand{\bepa}{\left\{ \begin{array}{l}}
\newcommand{\eepa} {\end{array}\right.}
\newcommand{\f}{\frac}
\newcommand{\p}{\partial}
\newcommand{\eps}{\epsilon}

\newcommand{\calC}{\mathcal{C}}

\newcommand{\calM}{\mathcal{M}}

\newcommand{\calP}{\mathcal{P}}

\newcommand{\ds}{\displaystyle}

\newcommand {\supp}{{\mathrm supp}}
\newcommand{\buK}{\mathbf{\underline{K}}}
\newcommand{\ddt}{\ds\f{d}{dt}}

\theoremstyle{plain}

\newtheorem{lemma}{Lemma}

\newtheorem{theorem}{Theorem}

\newtheorem{proposition}{Proposition}

\newtheorem{corollary}{Corollary}

\AtBeginEnvironment{subappendices}{%
\chapter*{Appendices}
\addcontentsline{toc}{chapter}{Appendices}
\counterwithin{figure}{section}
\counterwithin{table}{section}
}

\AtEndEnvironment{subappendices}{%
\counterwithout{figure}{section}
\counterwithout{table}{section}
}

\title{Selection-mutation dynamics with asymmetrical reproduction~kernels}

\author{
Beno\^ \i t Perthame\thanks{Sorbonne Universit\'e, CNRS, Universit\'e de
Paris, Inria, Laboratoire Jacques-Louis Lions, F-75005 Paris, France.}  
\thanks{B.P. has received funding from the European Research Council (ERC) under the European Union's Horizon 2020 research and innovation programme (grant agreement No 740623). }
\footnotemark[5] 
\and Martin Strugarek 
\thanks{AgroParisTech, 16 rue Claude Bernard, 75231 Paris Cedex 05,  France}
\footnotemark[1] \footnotemark[5]
\and
C\'ecile Taing\thanks{Laboratoire de Math\'ematiques et Applications, Universit\'e
de Poitiers, CNRS, F-86073 Poitiers, France.} \thanks{Emails: benoit.perthame@sorbonne-universite.fr,
 martin.strugarek@agriculture.gouv.fr, cecile.taing@math.univ-poitiers.fr}
}

\begin{document}
\maketitle

\begin{abstract}
 We study a family of selection-mutation models of a sexual population structured by a phenotypical trait. The main feature of these models is the asymmetric trait heredity or fecundity between the parents : we assume that each individual inherits mostly its trait from the female or that the trait acts on the female fecundity but does not affect male.   Following previous works inspired from principles of adaptive dynamics, we rescale time and assume that mutations have limited effects on the phenotype. Our goal is to study the asymptotic behavior of the population distribution. We derive non-extinction conditions and BV estimates on the total population. We also obtain Lipschitz estimates on the solutions of Hamilton-Jacobi equations that arise from the study of the population distribution concentration at fittest traits. Concentration results are obtained in some special cases by using a Lyapunov functional.
\end{abstract}

\noindent{\makebox[1in]\hrulefill}\newline
2010 \textit{Mathematics Subject Classification.}  35F21, 35B40, 35Q92, 45K05
\newline\textit{Keywords and phrases.} Integro-differential equations; Asymptotic analysis; Adaptive dynamics; Population biology;


\section{Introduction}

We study mathematically a family of models of selection-mutation for sexual populations structured 
with a continuous phenotype, which we call "trait" and denote by $x \in \R$, and we present 
different methods that apply to some specific cases. 
All models studied in the present 
paper are derived from the general form
\begin{equation}
\left\{
\begin{aligned}
& \eps\p_t n_\eps (t, x) = \frac{1}{\rho_\eps(t)}  \iint_{\R^2} K_\eps(x, y, z) n_\eps(t, y) n_\eps(t, z) dy \,dz - R (x, \rho_\eps(t)) n_\eps(t, x),
\\
& \rho_\eps(t) = \int_{\R} n_\eps(t, x) dx, \quad n_{\eps} (0, x) = n^0_{\eps} (x).
\end{aligned}
\right.
\label{eq:resgeneralform}
\end{equation}
The variable $t$ stands for time, $n_\eps(t, x) \in [0, +\infty)$ 
is the population number density at time~$t$ and with trait~$x$, and 
$\rho_{\eps} (t)$ is the total population. 
The positive function~$R$ represents the saturation term and comprises 
intrinsic mortality and the effects of competition through 
the nonlocal term $\rho_\eps$.
Indeed, we assume that all individuals compete 
for survival because they share the same resources, which
implies the boundedness of the total population. In this framework,
the integral term $\rho_\eps$ is refered as the competition term and $R$ is increasing 
with respect to this quantity.
Since we consider sexual population, the major feature of the
equations under study is to yield nonlinear and nonlocal 
birth terms with a quadratic aspect, though $1$-homogeneous. 
In equation \eqref{eq:resgeneralform}, we interpret~$y$ (the second 
argument for~$K_\eps$) as the female trait, and~$z$ (the third argument) as the male trait. 
Thus $x \mapsto K_\eps(x,y,z)$ is equal to the distribution of 
individuals that are born from any encounter between a female of trait $y$ 
and a male of trait~$z$, per unit of time.
Of course, this model is valid only assuming that the sex ratio is constant 
in time and independent of the trait. We make this simplification in order 
to obtain a single equation rather than a system.

Our specific motivation comes from insecticide resistance. 
This phenomenon has been observed among insects of interest 
for human health, in particular in species of mosquitoes that are 
vectors for diseases like dengue (in the {\itshape Aedes} genus) or malaria (in the {\itshape Anopheles} genus). 
For this specific problem of selection-mutation, the trait variable should contain, for instance, 
the expression level for the {\itshape kdr} gene ({\itshape knock-down resistance}, 
see~\cite{PasRay.1996}). The present study is part of a more general program 
on the analysis of models, and their control, in the context of evolutionary 
epidemiology (see~\cite{NadStrVau.Hindrances,SVZ} and the references therein).

Because of this motivation, our models have a sexual reproduction kernel. 
This is not the case in similar selection-mutation models developed for bacteria or resistance to 
treatment in cancer (see, e.g., \cite{Raoul_2011, Lorzetal2015}), where the reproduction is clonal. 
The same kind of kernels arise in various biological
problems, as cell alignment~\cite{DFR} or protein exchanges \cite{Biletal2016, MagRao2015},
and in more realistic models of trait-structured sexual populations (see \cite{Tufto2000, TurBar1994} for 
some examples in numerical frameworks).

\bigskip

The main results of this paper concern the behavior of~$\rho_\eps$ 
and~$n_\eps$ in the asymptotic of large time scale and mutations 
with limited effect on the phenotype. Because the general equation~\eqref{eq:resgeneralform} 
is out of reach with the methods we use here, inherited from 
asexual reproduction, we consider two particular  
classes of reproduction kernels in equation~\eqref{eq:resgeneralform}. 
They share the common property of an asymmetric structure which is biologically relevant.
Indeed, back to the insecticide resistance 
modeling, it has been observed that good resistance levels also
result in a high fitness cost, and especially on the fecundity. We simply assume
here that either fecundity is female-trait-dependent, or that new individuals
inherit mostly their trait from the female. Since
female mosquitoes have a longer lifespan than male ones, they will be more susceptible to 
be affected by insecticides and to become resistant. Also, females will perform 
several ovipositions during their lives, suggesting a higher impact of acquired resistance
on female fecundity.
\\

We consider a first class of models, with 
{\bfseries asymmetric fecundity} (AF in short),
\begin{equation*}\label{asfecundity}
\eps\p_t n_\eps (t, x) = \frac{1}{\rho_\eps(t)} \iint_{\R^2} B(y)\alpha_\eps(x,y,z) n_\eps(t, y) n_\eps(t, z) dy\, dz - R(x, \rho_\eps(t)) n_\eps(t, x),
\tag{AF}
\end{equation*}
where $B$ is a positive function and represents the crossing fecundity, 
which is assumed to depend only on female's trait, 
and $\alpha_\eps(\cdot, y, z)$ is the probability distribution 
of the offspring from a $y$ female and a $z$ male. Then the reproduction
kernel reads
\begin{equation}\label{af:hyp}
K_{\eps}(x, y, z) = B(y) \alpha_{\eps}(x, y, z), \quad \text{with }\int_\R \alpha_{\eps}(x, y, z) \,dx = 1 \text{ for all } y, z \in \R.
\end{equation}

The second class of models features an {\bfseries asymmetric trait heredity} (ATH in short), which reads
\begin{equation*}\label{astrait}
\eps\p_t n_\eps (t, x) = \frac{1}{\rho_\eps(t)} \iint_{\R^2} K_0(x-z)G_\eps(x-y) n_\eps(t, y) n_\eps(t, z) dy\, dz - R(x, \rho_\eps(t)) n_\eps(t, x),
\tag{ATH}
\end{equation*}
where
\begin{equation}\label{ath:hyp}
K_\eps(x,y,z)=K_0(x-z)G_\eps(x-y),
\end{equation}
with $K_0$ a positive function, and $G_\eps$ the rescaling of a positive 
function~$G$ by letting
\begin{equation*}
G_\eps(x-z)=\frac{1}{\eps} G\left( \frac{x-z}{\eps}\right), \quad \text{with} \quad\int_\R G(z) dz=1.
\end{equation*}
We can write the ATH equation under the following form
\begin{equation*}
\eps\p_t n_\eps (t, x) = \frac{1}{\rho_\eps(t)} \left[K_0 \ast n_\eps(t, \cdot) \, G_\eps \ast n_\eps(t, \cdot)\right] (x)   - R(x, \rho_\eps(t)) n_\eps(t, x).
\end{equation*}

For some particular forms of these two classes of models, we use three ingredients, inspired from methods
used in asexual population models, to 
state convergence results. Firstly, we derive 
some Bounded Variation 
($BV$ in short) estimates for~$\rho_\eps$. Secondly,
we prove concentration of the population by a Lyapunov stability property.
Finally,we identify 
a consistent limit object as $\eps \to 0$, which is a constrained 
Hamilton-Jacobi equation, and we obtain compactness estimates 
on the solutions at the $\eps$-level in order to be able to extract 
converging subsequences and to use the stability property of viscosity 
solutions. 
\\

To better show the technical ideas and highlight 
the new arguments, we begin with studying two simplified 
models which are  particular cases of the two classes presented 
above. The simplifications are that
\\
$\bullet$ we ignore mutations, therefore  the dynamics 
is simply generated by adaptation and competition between 
pre-existing traits,
\\
$\bullet$ we assume that the saturation function~$R$ 
does not depend on the trait variable and is such that 
$R (x, \rho) \equiv \nu \rho$, with $\nu>0$ which allows 
for some  specific algebraic manipulations.

The model with no mutations reads
\begin{equation*}\label{nomutation}
\eps\p_t n_\eps (t, x) = \left( \frac{1}{\rho_\eps(t)} K_0 \ast n_\eps(t, \cdot) (x) - \nu \rho_\eps(t) \right) n_\eps(t, x),
\tag{nM}
\end{equation*}
This equation can be written under the form of equation~\eqref{asfecundity} with
\begin{equation*}
B \equiv \int_\R K_0(z) \, dz, \text{ and } \alpha_\eps(x,y,z)=\frac{1}{B}K_0(x-z)\delta_0(x-y),
\end{equation*}
and also under the form of \eqref{astrait} with
\begin{equation*}
G_\eps=\delta_0.
\end{equation*}
Because it is very specific, we also introduce a generalization of \eqref{nomutation}  and consider more general birth and competiton
\begin{equation*}
 \eps \p_t n_{\eps} (t, x) = \Big( \f{1}{\rho_{\eps}(t)} \int K_S(x,y) n_{\eps} (t,y) dy - \big( R_0 (x) + R_1 (\rho_{\eps}) \big) \Big) n_{\eps}(t,x), 
 \label{eq:gennLyapunov}
 \tag{gnM}
\end{equation*}
for some symmetric kernel $K_S : \R^2 \to \R_+$.  With this generality, we  show in Section~\ref{ch3:subs:Lyapunov} how to built a Lyapunov convergence results for \eqref{eq:gennLyapunov} based on tools from game theory. The method is based on the reduction to some kind of replicator equation for the quantity $q(t,x) :=\frac{ n_{\eps} (\eps t, x)}{\rho_{\eps}(\eps t)},$ which satisfies another equation in closed form.

\bigskip

The relationships between sexual reproduction and selection are not well understood. 
Models of sexual reproduction have already been discussed in different contexts. 
Studies of individual-based models of sexual population were performed to determine 
the necessary conditions to evolutionary branching in \cite{DieckDoe1999, KisGer1999, DooDieck2006}, 
with a structure in genetic types (see \cite{bulmer, burger} for a review of 
mathematical models of population genetics).
In \cite{ColMelMet2013} for instance, the authors investigate a stochastic birth 
and death process model for sexually reproducing diploids with Lotka-Volterra type 
dynamics and single locus genetics. At the small mutation steps limit, they derive 
a differential equation in allele space, referred to as a form of the canonical 
equation of the adaptive dynamics. In~\cite{Coretal2018}, another stochastic 
birth and death process model is studied with sexual reproduction according to 
mating preferences and a space structure with patches. In this case, reproductive 
isolation between patches occurs, and the authors prove that the time needed for 
this isolation to occur is a function of the population size. In the framework 
of insecticide resistance, a deterministic system with three genotypes (two alleles 
at a single locus) was studied in~\cite{SchSou2015}, with a focus on the ``reversal time'' 
that is a measure of the persistence of resistance in a population after exposition to insecticide.

From a large population point of view, in~\cite{MirRao} the authors considered 
sexual populations structured by a trait and a space variable 
in a non-homogeneous environment, and after performing an 
asymptotic limit and a simplification of the model, they derived 
an estimate of the invasion speed or extinction speed of the 
population. In \cite{Bourgeron2017}, the authors study the same 
kind of models as in the present paper, where the traits of 
the newborns are distributed through a gaussian kernel centered 
on the mean of the parents' traits and with a constant variance, 
as in \cite{Doeetal2007}, which is the so-called infinitesimal 
model. They prove the existence of principal 
eigenelements for the corresponding eigenproblem, using the 
Schauder fixed point theorem. This work has been extended in 
\cite{calvez2018asymptotic} with the study of the same stationary
problem at the asymptotic of vanishing variance. In the same regime,
the associated Cauchy problem has then been investigated in \cite{patout2020cauchy},
showing that solutions can be approximated by Gaussian profiles with
small variance.

\bigskip

The paper is organized as follows. In Section~\ref{ch3:mainresults}, we state 
our assumptions and results. We also establish some 
non-extinction conditions and bounds on the total population. In 
Section~\ref{sub:simplecase}, we focus on the models without mutations \eqref{nomutation}-\eqref{eq:gennLyapunov} in order 
to introduce the main arguments that will be used for the more general cases. In 
particular, we derive $BV$ estimates for the total population, and prove a Lyapunov 
stability result for the population distribution. In Section~\ref{section:mutations}, we 
address the derivation of $BV$ estimates for the \eqref{astrait} and \eqref{asfecundity} models when $R$ only 
depends on the total population variable and we explain the difficulties 
encountered when $R$ is generic. Also, discuss the settings 
of the Lyapunov method applied to these mutation models. 
Finally in Section~\ref{ch3:HJ}, we deal with the 
Hamilton-Jacobi approach. To conclude we identify some difficulties 
raised by the application of our methods to the general case 
of~\eqref{eq:resgeneralform} and other possible approaches.

\section{Main results} \label{ch3:mainresults}

 In order to introduce our main results, we need several assumptions where we use the following notations. 
 \\
We denote by $\calM_+^1 (\R)$ the set of probability measures on $\R$, by $\calM_+ (\R)$ the set of finite mass nonnegative measures and by $\mathcal{C}_b (\R, \R_+)$ the space of continuous and bounded functions on $\R$ with values in $\R_+$. Also, for $a \in \R$, we use the notation $a_- = \max(-a, 0)$.

\subsection{Assumptions and statements}

The initial data is denoted by $ n_{\eps}^0 (x) $ and, to prove a $BV$ bound on  $\rho_{\eps}$, which plays a fundamental role hereafter, the inital total density is  usually assumed to satisfy
\begin{equation}\label{assum:rho_0}
\eps (\dot{\rho}_{\eps})_-(0)  \quad \text{ is uniformly bounded}.
\end{equation}
We take the value $ (\dot{\rho}_{\eps})_-(0)$ from the equation under consideration. For instance, for the model~\eqref{nomutation}, it is defined by 
\[
\eps \dot{\rho}_{\eps}(0)  = \int n_{\eps}^0 (x) 
\frac{K_0 \ast n_{\eps}^0}{\rho_{\eps}^0} (x)\, dx - \nu (\rho_{\eps}^0)^2  .
\]

The function $R$ stands for the death rate and the competition effects. 
We make the standard assumption that it increases with the total population:
\begin{equation}
\forall x, \rho, \quad \p_{\rho} R(x, \rho) > 0.
\label{hyp:Rinc}
\end{equation}
%
\medskip
For models with no mutations \eqref{nomutation} and asymmetric trait heredity \eqref{astrait}, we assume
\begin{equation}\label{assum:K0}
K_0 \in \mathcal{C}_b (\R, \R_+) \text{ is an even kernel}, \quad K(z)=K(-z).
\end{equation}
The symmetry is not always needed but the positivity of the symmetric part of the kernel  is fundamental, this is why we underline this property. Continuity is needed because the kernel acts on measures in the limit when $\eps$ vanishes.

\medskip

For  equation~\eqref{nomutation}, we state the following theorem which is proved in Section~\ref{sec:BVsimple}.

\begin{theorem}[$BV$ bound for model \eqref{nomutation}]\label{th:simple}
We assume \eqref{assum:K0} and let $n_{\eps}$ be the solution of~\eqref{nomutation} with an initial data $n^0_{\eps}$ satisfying  \eqref{assum:rho_0}.

Then, for all $T>0$, $\rho_\eps$ is uniformly bounded in $BV(0, T)$. Namely, 
we obtain
\begin{equation}
    \int_0^T \lvert \dot{\rho}_{\eps} (t) \rvert \,dt \leq \rho_M + \frac{2 \eps}{\kappa''_m} (\dot{\rho}_{\eps})_-(0)
, \quad  \int^T_0 \int_\R n_\eps \big( \frac{K_0 \ast n_\eps}{\rho_\eps} - \nu\rho_\eps \big)^2  dx\,dt=  O(\eps), 
 \label{eq:concentration}    
\end{equation}
with $\rho_M$ and $\kappa''_m$ defined later on. This implies that, up to extraction of subsequences, there exist limits $\rho_{\eps} \to \rho$ in $L^1(0,T)$, and $n_{\eps} \xrightharpoonup{} n  \in L^{\infty}_t (0,T;\calM_+ (\R))$ in the sense of measures.
\end{theorem}
The first bound in~\eqref{eq:concentration} gives compactness in $L^1_{\rm loc}$ for $\rho_{\eps} (t)$, which is useful for nonlinear terms. Formally, the second bound in~\eqref{eq:concentration}   means that the support on the limit measure $n$ is supported by the points $\bar x(t)$ where $K_0 \ast n(t,x)= \nu \rho(t)$ which are expected to be discrete (if not unique for all $t$). This question is studied in Section~\ref{sec:CDM}.  

\medskip

For the model with asymmetric fecundity \eqref{asfecundity}, we need 
the following assumption on $B$ and~$\alpha$:
\begin{equation}
\begin{aligned}
&\exists C >0, \forall \eps > 0, \forall \phi \in \calM_+^1 (\R), \\
&\iiint_{\R^3} \alpha_{\eps} (x, y, z) B(x) B(y) \phi(y) \phi(z) dx \,dy \,dz - \left( \int_\R B(y) \phi(y)dy \right)^2 \geq - C \eps.
\end{aligned}
\label{assum:alphae}
\end{equation}
Firstly, when  $B$ is constant,  this assumption is obviously satisfied. Secondly, for $\phi$ a Dirac mass at $x_M$, this  assumption reduces to 
\[
\int_{\R} \alpha_{\eps} (x, x_M, x_M) B(x) dx - B(x_M) \geq - \frac{C}{B(x_M)} \eps,
\]
Recalling that $\int \alpha_{\eps} (x, y, z) dx = 1$ for all $y, z$, this 
implies that as $\eps$ vanishes, $\alpha_{\eps} (\cdot, x_M, x_M)$ 
is concentrated at points where $B$ is equal to its maximum $B(x_M)$, 
which is a restrictive necessary condition for \eqref{assum:alphae} to hold.
Thirdly, we state a sufficient condition: if $\alpha_{\eps} (\cdot, y, z) \to \alpha_0 (y, z) \in \mathcal{M}^1_+ (\R)$ with either
\[
 \forall \, y, z, \quad \int_\R \alpha_0 (y, z) (x) B(x) dx \geq B(y),
\]
or
\[
 \forall \, y, z, \quad \int_\R \alpha_0 (y, z) (x) B(x) dx \geq B(z),
\]
and if convergence is sufficiently fast, then~\eqref{assum:alphae} holds. 
In the first case this is a consequence of the Cauchy-Schwarz inequality, 
and in the second case we simply obtain that the left-hand side 
in~\eqref{assum:alphae} converges to~$0$ as~$\eps$ vanishes.
In particular, we may assume 
$\alpha_{\eps} (x, y, z) = \frac{1}{\eps} G \big( \frac{x - y}{\eps} \big)$ 
or $\frac{1}{\eps} G \big( \frac{x - z}{\eps} \big)$ for some appropriate 
kernel $G$. These situations are those we have in mind, although~\eqref{assum:alphae} in all generality may allow for some other cases. 
\\
All in all, \eqref{assum:alphae} means that the fecundity is improved from that of parents with the same trait. More generally,  the fecundity variation from one generation to the next is controlled from below by that of the parents.
Unsurprisingly, 
this dissipative feature implies that the variations of $\rho_\eps$ can 
be controlled at the limit $\eps \to 0$, as stated in  the following result whose proof is given in Section~\ref{ch3:estimBV}.

%
\begin{proposition}[$BV$ bound for \eqref{asfecundity}]\label{th_asfecundity}
    
      Let $n_{\eps}$ be the solution of \eqref{asfecundity} with initial data $n^0_{\eps}$ satisfying~\eqref{assum:rho_0}. Assume that $R(x,\rho)=\nu \rho$ and \eqref{assum:alphae}.
     
    Then, for all $T>0$, $\rho_\eps$ is uniformly bounded in $BV(0, T)$ and we have
\begin{equation*}
    \int_0^T \lvert \dot{\rho}_{\eps} (t) \rvert \,dt \leq \rho_M +\frac{ 2\eps}{\nu\rho_m} (\dot{\rho}_{\eps})_-(0)  
+\frac{2C}{\nu\rho_m}\left(T+\frac{\eps}{\nu\rho_m}(e^{-\frac{\nu\rho_m T}{\eps}}-1)\right),
\end{equation*}    
with $C$, $\rho_M$ and $\rho_m$ defined later on. This implies that, up to extraction of subsequences, there exist limits $\rho_{\eps} \to \rho$ in $L^1(0,T)$, and $n_{\eps} \xrightharpoonup{} n \in L^{\infty}_t (0,T;\calM_+(\R))$ in the sense of measures.
\end{proposition}

\bigskip

In order to apply the same technique as for the model without mutations 
addressed in Section~\ref{sub:simplecase}, we need a convergence  assumption on $G_\eps$ as $\eps$ vanishes. More precisely,
we assume that there exists $C \in \R_+^*$ such that 
\begin{equation}\label{hyp:conv_ge}
\begin{aligned}
&\forall \eps > 0, \forall \phi \in W^{1,1}\text{ with }\lVert \phi' \rVert_{L^1} \leq 1, \forall \psi \in L^{\infty}\text{ with }\lVert \psi \rVert_{L^{\infty}} \leq 1, \\
&\Big\lvert \int_\R \psi(x) ( G_{\eps} \ast \phi ) (x) dx - \int_\R \psi(x) 
\phi(x) dx \Big\rvert \leq C \eps.
\end{aligned}
\end{equation}
This assumption on the convergence of $G_{\eps}$ as $\eps$ vanishes  
holds in the typical case where $G_{\eps}$ is Gaussian with variance $\eps^2$.
Specifically, we write $G_{\eps} (x) = \frac{1}{(2 \pi \eps^2)^{1/2}}e^{-x^2 / 2 
\eps^2}$.


We obtain the following result whose proof is given in Section~\ref{ch3:estimBV}.
\begin{proposition}[$BV$ bound for \eqref{astrait}]\label{th3}
Let $n_{\eps}$ be the solution to \eqref{astrait} associated with initial 
data $n^0_{\eps}$ satisfying~\eqref{assum:rho_0}. Assume~\eqref{assum:K0}, \eqref{hyp:conv_ge}, that that $K_0$ belongs to  $W^{1, 1}$ and the following ("non-extinction" in this case) condition
\begin{equation}
\exists \eta_0 > 0, \quad \forall \eps > 0, \quad \eta_{\eps} :=\inf_{\phi \in \calM_+^1(\R)} \, \int K_0 \ast \phi \cdot G_\eps \ast \phi \,dx \geq \eta_0.
\label{cond:nonextC4}
\end{equation}

Then $\rho_\eps$ is uniformly bounded in $BV(0, T)$. Namely, we have
\begin{equation*}
\int_0^T \lvert \dot{\rho}_{\eps} (t) \rvert \,dt \leq \rho_M+2 (\dot{\rho}_{\eps} (0))_- \frac{\eps}{C_1}(1-e^{-C_1 T / \eps}) + 2\frac{\eps C_2}{C_1^2} ( e^{-C_1 T / \eps} -1)+2 \frac{C_2}{C_1}T.
\end{equation*}
Then, up to extraction there exist $\rho\in L^1_{loc}(0,\infty)$ and $n \in 
L^{\infty}_t (0,T;\calM_+ (\R))$ such that $(\rho_{\eps})$ converges towards~$\rho$ in 
$L^1_{loc}(0,\infty)$, and $(n_{\eps})$ towards $n$ in the sense of measures, 
when~$\eps$ vanishes.

Moreover, for all $T > 0$, we have
\begin{equation*}
\int_0^T  \hskip-4pt  \int_{\R} (G_\eps \ast n_\eps) \left[\frac{K_0 \ast n_\eps}{\rho_\eps}- \nu \rho_\eps\right]^2dx\,dt=O(\eps).
\end{equation*}
\end{proposition}

For the generalized no mutation model \eqref{eq:gennLyapunov}, 
 a Lyapunov structure is identified under the following assumptions on $K_S : \R^2 \to \R_+$, $R_0 : \R \to \R_+$ 
and $R_1 : \R_+ \to \R_+$ :
\begin{align}
 &K_S \in \calC_b (\R^2, \R_+) \text{ is symmetric: } \forall x, y \in \R, \, K_S (x,y) = K_S (y,x),
 \label{assumption:Ksym}
 \\
 &\forall \xi \in \calM_+ (\R) \backslash \{0 \}, \quad \iint_{\R^2} K_S(x,y) \xi(x)\xi(y) dx dy > 0,
 \label{assumption:Kpos}
 \\
 &\supp(q^0) \text{ is compact or } R_0 \text{ is proper}, \quad \text{with } q^0= \frac{n^0_\eps}{\rho_\eps^0}, 
 \label{assumption:R0}
 \\
 &R_1 \text{ is increasing and proper},
 \label{assumption:R1}
 \\
 &\exists ! \, x_M \in \supp(q^0), \quad y \mapsto K_S(x_M, y) - R_0 (y) \text{ reaches its maximum at } x_M.
 \label{assumption:xM}
\end{align}
In this framework, a stability result is obtained for the population density $\delta_{x_M}$.
%
\begin{theorem} [Local stability of $x_M$]  \label{thm:lyapunov}
 Under assumptions \eqref{assumption:Ksym}-\eqref{assumption:xM}, the Dirac mass $\delta_{x_M}$ is locally asymptotically stable for~\eqref{eq:gennLyapunov}.
\end{theorem}

This result is based on the construction of a Lyapunov functional which has been advocated for these problems of adaptive dynamics, e.g., in~\cite{JabRao2011}. Here it relies on the particular structure of~\eqref{eq:gennLyapunov}, which can be reduced to a continuous replicator equation. Then, we can apply game theoretical methods  following \cite{Cheung2014, Cheung2016,Sandholm2001,Sandholm2010}.
\\

In the general case of a death rate depending on both traits and 
the total population, Lyapunov functionals are not available and 
the methods Theorem~\ref{thm:lyapunov} do not apply. Therefore, 
following \cite{Dieketal2005, BarPer2008, LorMirPer2011}, 
we may try to express concentration of $n_\eps$ at a point $\bar x$ 
as in a low temperature Gaussian $\frac{1}{\sqrt{2 \pi \eps}}\exp(-\frac{|x- \bar x|^2}{2\eps})$.  
Because the quadratic form is too specific for our problem, we rather perform the Hopf-Cole transform 
\begin{equation*}
u_\eps(t,x)=\eps \ln n_\eps (t,x),
\end{equation*}
and apply a Hamilton-Jacobi approach.  
The limiting function $u= \lim_{\eps \to 0} u_\eps$ will give the concentration shape analogous to $-\frac{|x- \bar x|^2}{2}$ but specific to the problem at hand.  In particular, the population concentrates on the points where $u(t,x)$ vanishes. When there is a unique point, monomorphism occurs but  polymorphism is possible. The existence of such a limit is asserted by the

\begin{theorem}[Lipschitz estimates for $u_\eps$]\label{th:HJ}
Under some assumptions on the initial data $u_\eps^0$, for both models \eqref{asfecundity} and \eqref{astrait}, the corresponding $u_\eps$ are locally Lipschitz uniformly in $\eps$.

Moreover, we have a global upper bound on $u_\eps$. Namely, there exists a constant $C$, such that
\begin{equation*}
u_\eps(t,x) \leq \eps \ln \left(C+\frac{C(1+t)}{\eps}\right).
\end{equation*}

Consequently, we can extract from $u_\eps$ a sequence which converges locally uniformly to a limit $u(t,x) \leq 0$ and the limiting concentration points of $n_\eps$ are included in the set $\{u(t,x)=0\}$.

\end{theorem}

The proof of this theorem and its consequences are the topic of Section~\ref{ch3:HJ}.
It requires specific assumptions in both cases (AF or ATH),  which are too long for this general presentation, and thus are specified in the corresponding sections.

\subsection{Boundedness of the total population and non-extinction}

In preparation to prove these theorems, we begin with some controls of the total population, $\rho_\eps$, which, in full generality  satisfies
\begin{equation} \label{eq:GenRho}
\eps\dot{\rho}_\eps(t) = \bigintssss_\R  \Big( \iint_{\R^2} K_\eps(x,y,z)   \frac{n_\eps(t,z)}{\rho_\eps(t)}n_\eps(t, y) dy\,dz  - R(x, \rho_\eps(t))n_\eps(t, x)\Big) \,dx.
\end{equation}
We define
\begin{equation}
K_M := \sup_{0<\eps \leq 1}\sup_{\phi \in \calM_1^+ (\R)} \, \sup_y \, \iint_{\R^2} K_\eps(x,y,z) dx\, \phi(z) dz < +\infty,
\label{hyp:Kbounded}
\end{equation}
and, to ensure that $\rho_\eps$ remains bounded along all trajectories, we complement \eqref{hyp:Rinc} with 
\begin{equation}
\label{hyp:Rinf}
\begin{aligned}
&\exists R_m : \R_+ \to \R_+, \text{ increasing, with } R_m(0) = 0, \quad R_m (+\infty) = +\infty, \\
& \text{and}\quad \forall x, \; R(x, \rho) \geq R_m (\rho), \qquad \rho_M := R_m^{-1}(K_M). 
\end{aligned}
\end{equation}

 We first observe the following boundedness result:
\begin{proposition}[Upper bound for $\rho_\eps$]\label{rho:upperbound}
Under assumptions \eqref{hyp:Rinc}, \eqref{hyp:Rinf} and \eqref{hyp:Kbounded}, all trajectories of \eqref{eq:resgeneralform} are forward-$\rho_M$-bounded from above in $\rho_\eps$, by which we mean that $\dot{\rho}_\eps (t) < 0$ as long as $\rho_\eps(t) > \rho_M$.
\end{proposition}
Indeed, using the equation  \eqref{eq:GenRho} and the assumptions \eqref{hyp:Rinf} and \eqref{hyp:Kbounded}, we may write
\[
\eps\dot{\rho}_\eps(t) \leq \rho_\eps(t) [K_M - R_m (t{\rho}_\eps(t) )],
\]
from which the result follows immediately. 
\\

Conversely, we can study conditions that ensure non-extinction of the population: $\rho_\eps(t) \geq \rho_m > 0$. As a first example, let
\begin{equation}
\kappa_m (\rho) := \inf_{0<\eps \leq 1}\inf_{\phi \in \calM_+^1 (\R)} \, \inf_y \,   \iint_{\R^2}  K_\eps(x,y,z) dx \, \phi(z) dz - R(y, \rho)
\label{hyp:Kinf}
\end{equation}
\begin{proposition}[Lower bound for $\rho_\eps$ under assumption \eqref{hyp:Kinf}]
Under assumption \eqref{hyp:Rinc} and if there exists $\rho_m > 0$ such that $\kappa_m (\rho_m) = 0$, with $\kappa_m$ defined in~\eqref{hyp:Kinf}, then all trajectories of~\eqref{eq:resgeneralform} are forward-$\rho_m$-bounded from below in $\rho_\eps$, by which we mean that $\dot{\rho}_\eps (t) > 0$ as long as $\rho_\eps(t) < \rho_m$.
\label{prop:nonext}
\end{proposition}
 This again follows from  \eqref{eq:GenRho}, writing this time the lower control
\[
\eps\dot{\rho}_\eps(t) \geq \kappa_m (\rho) ,
\]
and remembering that $\p_\rho R(x,\rho)\leq 0$, we infer that $\kappa_m$ is non-decreasing which gives $\dot{\rho}_\eps(t) \geq 0$ for $\rho \leq \rho_m$.
\\

However, $\kappa_m (0) > 0$ is not expected to be a necessary condition. It is an open and challenging question to determine more general conditions for non-extinction, and study the set of extinction trajectories in cases when these conditions are not met.

For instance, a second non-extinction result is
\begin{equation}
\label{hyp:kappa1}
\kappa'_m (\rho) := \inf_{0<\eps \leq 1} \inf_{\phi \in \calM^1_+(\R)} \int_{\R} \big( \iint_{\R^2} K_\eps(x, y, z) dx \, \phi(z) dz  - R(y, \rho) \big) \phi(y) dy.
\end{equation}

\begin{proposition}[Lower bound for $\rho_\eps$ with a condition on \eqref{hyp:kappa1}]
Under assumption~\eqref{hyp:Rinc} and if there exists $\rho_m > 0$ such that $\kappa'_m (\rho_m) = 0$ then all trajectories of~\eqref{eq:resgeneralform} are forward-$\rho_m$-bounded from below in $\rho_\eps$.
\label{prop:nonext2}
\end{proposition}
 This time we factorize $\rho_\eps$ in the integral of  \eqref{eq:GenRho}, and using the special choice $\phi = \frac{n_\eps}{\rho_\eps}$ we may write the lower control
\[
\eps\dot{\rho}_\eps(t) \geq  {\rho}_\eps(t) \iiint_{\R^3} K_\eps(x, y, z) dx \, \phi(z)  \phi(y) dx dy dz  - {\rho}_\eps(t) \int_{\R} R(x, \rho) \phi(x) dx \geq  {\rho}_\eps(t) \kappa'_m (\rho_\eps).
\]
Then, we conclude as before..
\\

And likewise,  assume that
\begin{equation}
\label{hyp:Rsup}
\begin{aligned}
\exists R_M : \R_+ \to \R_+, \text{ increasing, with } R_M(0) \geq 0, \\
R_M (+\infty) = +\infty \text{ and } \forall x, \quad R(x, \rho) \leq R_M (\rho),
\end{aligned}
\end{equation}
and
\begin{equation}
\kappa''_m := \inf_{0<\eps \leq 1} \inf_{\phi \in \calM^1_+(\R)} \iiint_{\R^3} K_\eps(x, y, z) \phi(y) \phi(z) \,dx \, dy \, dz > R_M(0).
\label{eq:kappa2m}
\end{equation}
Then, we have the
\begin{proposition}[Lower bound for $\rho_\eps$ with condition \eqref{eq:kappa2m}]
Assume \eqref{hyp:Rsup}  and \eqref{eq:kappa2m}. Then all trajectories of~\eqref{eq:resgeneralform} are forward-$\rho_m$-bounded from below in $\rho_\eps$, with $\rho_m = R_M^{-1} (\kappa''_m)>0$.
\label{prop:nonext3}
\end{proposition}
This result follows from the same type of calculation, writing 
\[
\dot{\rho_\eps} \geq \big( \kappa''_m - R_M(\rho_\eps) \big) \rho_\eps .
\]

\section{The models without mutations}
\label{sub:simplecase}

In order to see clearly the kind of results to be expected, we first study in details a very simple example, which is equation \eqref{nomutation}.
The form of the birth rate assumes that the trait is perfectly transmitted from the females to their progeny, and the cross-fecundity between a male of trait $z$ and a female of trait $x$ depends only on the distance between $x$ and $z$ through $K_0$.

Assumptions~\eqref{hyp:Rinc} and \eqref{hyp:Rinf} (for $R_m = \nu \rho$) obviously hold in case \eqref{nomutation}. Assumption \eqref{hyp:Kbounded} holds with $K_M = \max_x K_0(x)$. Therefore we can apply Proposition~\ref{rho:upperbound} and the total population remains bounded $\rho_\eps (t) \leq \rho_M$.

We can also verify non-extinction. Since $\kappa_m (\rho) = \inf K_0 -\rho$ the non-extinction condition from Proposition~\ref{prop:nonext} holds if and only if $\inf_x K_0(x) > 0$. However, even when  $\inf_x K_0 (x) = 0$, following  \cite{JabRao2011} in the context of entropy-based stability, we may assume that $K_0$ is such that 
\[
    \kappa''_m = \inf_{\phi \in \calM^1_+(\R)} \int_{\R} \big( K_0 \ast \phi \big) (x) \phi(x) \,dx > 0,
\]
and then, the assumptions of Proposition~\ref{prop:nonext3} holds which gives a lower bound on $\rho_\eps$

\subsection{Proof of Theorem \ref{th:simple} and {\em BV} estimates}
\label{sec:BVsimple}

We begin with  proving the $BV$ bound, which is the main result from which the other follow easily. To do so, we follow the strategy which is to first control the decay rate of $\rho_\eps$ as elaborated in~\cite{BarPer2008}. 
We depart from equation~\eqref{nomutation}, that we reformulate as 
\begin{equation} \label{eq:BVfund}
\eps \dot{\rho_\eps} (t) = \int_{\R} n_\eps(t, x) \frac{K_0 \ast n_\eps(t, \cdot)}{\rho_\eps(t)} (x) \,dx - \nu\rho_\eps^2.
\end{equation}
In order to differentiate it, we first use assumption~\eqref{assum:K0} which yields 
\[
\frac{d}{dt} \int_{\R} n K_0 \ast n = 2 \int_{\R}  n K_0 \ast (\p_t n),
\]
and consequently, we find from equation \eqref{eq:BVfund}
\[
\eps \ddot{\rho}_\eps = - \nu\rho_\eps \dot{\rho}_\eps - \nu\rho_\eps \dot{\rho}_\eps - \frac{\dot{\rho}_\eps}{\rho_\eps^2} \int_{\R} n_\eps K_0 \ast n_\eps + \frac{1}{2 \rho_\eps} \frac{d}{dt} \int_{\R} n_\eps K_0 \ast n_\eps + \frac{1}{\rho_\eps} \int_{\R} \p_t n_\eps K_0 \ast n_\eps.
\]
We rewrite this as
\begin{align*}
\eps \ddot{\rho}_\eps = - \nu\rho_\eps \dot{\rho}_\eps - 
\frac{\dot{\rho}_\eps}{2\rho_\eps^2} \int_{\R} n_\eps K_0 \ast n_\eps + \frac{1}{2} &
\frac{d}{dt} \Big(\frac{1}{\rho_\eps} \int_{\R} n_\eps K_0 \ast n_\eps - \nu\rho_\eps^2 \Big) 
\\& + \frac{1}{\eps} \int_{\R} \big( \frac{ n_\eps (K_0 \ast n_\eps)^2}{\rho_\eps^2} - \nu 
n_\eps K_0 \ast n_\eps \big).
\end{align*}
Inserting the equation \eqref{eq:BVfund} in this equality, we get
\begin{equation}
\frac{\eps}{2} \ddot{\rho}_\eps =  - \frac{\dot{\rho}_\eps}{2 \rho_\eps^2} \int_{\R} n_\eps K_0 \ast n_\eps + \frac{1}{\eps} \int_{\R} n_\eps \big( \frac{K_0 \ast n_\eps}{\rho_\eps} - \nu\rho_\eps \big)^2.
\label{eq:simpleddrho}
\end{equation}

Several conclusions follow from \eqref{eq:simpleddrho}. Firstly, $\ddot{\rho}_\eps \geq - \frac{\dot{\rho}_\eps}{\eps \rho_\eps^2} \int_{\R} n_\eps K_0 \ast n_\eps$, hence if $\dot{\rho}_\eps = 0$ then $\ddot{\rho}_\eps \geq 0$. In particular, $\rho_\eps$ has no strict local maximum. We can conclude that $\rho_\eps$ is either decreasing, increasing or decreasing-increasing, and since it is bounded, $\rho_\eps(t)$ must converge to some finite value $\rho_\eps^{\infty}$ as $t$ goes to $+\infty$.

Secondly, let $b_{\eps}(t) := \frac{1}{\rho_{\eps}^2(t)} \int_{\R} n_{\eps} (t, x) (K_0 \ast n_{\eps} (t, \cdot) ) (x) dx \geq \kappa''_m > 0$. Then from~\eqref{eq:simpleddrho},
\[
\frac{d}{dt} (\dot{\rho}_{\eps})_- \leq - \frac{\kappa''_m}{\eps} (\dot{\rho}_{\eps})_-.
\]
As a consequence we can control the decay of $\dot{\rho}_\eps$ thanks to the inequality  $(\dot{\rho}_{\eps})_- (t) \leq e^{-\frac{\kappa''_m t}{\eps}} (\dot{\rho}_{\eps})_- (0)$ which is bounded thans to assumption~\eqref{assum:rho_0}.
Next, to control the $BV$ norm, we use the upper bound $\rho_M$ mentioned above, and write
\begin{align*}
    \int_0^T \lvert \dot{\rho}_{\eps} (t) \rvert \,dt &\leq \int_0^T \dot{\rho}_{\eps} (t) \,dt + 2 \int_0^T (\dot{\rho}_{\eps})_- \,(t) dt
    \\
    &\leq \rho_M + 2 (\dot{\rho}_{\eps})_-(0) \int_0^T e^{- \frac{\kappa''_m t}{\eps}} dt
    \\
    &\leq \rho_M + 2 (\dot{\rho}_{\eps})_-(0) \frac{\eps}{\kappa''_m} \big( 1 - e^{-\frac{\kappa''_m T}{\eps}} \big).
\end{align*}
Therefore, under the mild assumption~\eqref{assum:rho_0} on the initial data, the family $(\rho_{\eps})_{\eps}$ is uniformly bounded in $BV(\R_+)$.
\\

We can now establish the second bound  in~\eqref{eq:concentration}. Going back to equation~\eqref{eq:simpleddrho}, and integrating it over $[0,T]$ for~$T>0$, we obtain
\begin{equation}
 \int^T_0\int_\R n_\eps \big( \frac{K_0 \ast n_\eps}{\rho_\eps} - \nu\rho_\eps \big)^2  dx\,dt=  \eps\int_0^T\frac{\dot{\rho}_\eps}{2 \rho^2} \int_{\R} n_\eps K_0 \ast n_\eps \,dx\,dt+ \frac{\eps^2}{2} (\dot{\rho}_\eps(T)-\dot{\rho}_\eps(0)).
\end{equation}
 Since we have already proved that $\rho_\eps$ is locally $BV$, uniformly in $\eps$, and using the assumptions  \eqref{assum:rho_0} and \eqref{assum:K0}, we deduce that
\begin{equation*}
 \int^T_0\int_\R n_\eps \big( \frac{K_0 \ast n_\eps}{\rho_\eps} - \nu\rho_\eps \big)^2  dx\,dt=  O(\eps),
\end{equation*}
which is the second bound in \eqref{eq:concentration}.

The other conclusions of Theorem \ref{th:simple} are  standard  function analytic consequences.

\subsection{Concentration in Dirac masses}
\label{sec:CDM}

We now comment on the consequences on the second bound  in~\eqref{eq:concentration}. 
Formally, at the limit $\eps \to 0$, the previous estimate yields
    \begin{equation}
    \int_{\R} n(t, x) \Big(\frac{K_0 \ast n(t,\cdot)}{\rho(t)}(x) - \nu\rho(t) \Big)^2 dx = 0.
    \label{eq:formal_concentration}
    \end{equation}
We may try to find  admissible solutions of \eqref{eq:formal_concentration} under the form of combinations of Dirac masses
\[
n = \sum_{i = 1}^N \rho_i \delta_{x_i},  \qquad \rho_i > 0, \quad \text{with} \quad\sum_{i=1}^N \rho_i = \rho.
\]
Inserting this expresion in \eqref{eq:formal_concentration}, we obtain
\[
    \sum_{i=1}^N \rho_i \Big(\sum_{j=1}^N \frac{\rho_j}{\rho} K_0 (x_i - x_j) - \nu\rho \Big)^2 = 0.
\]
In other words, we need to impose 
\begin{equation}
    \sum_{j=1}^N \frac{\rho_j}{\rho} K_0 (x_i - x_j) = \nu\rho \qquad \forall \; i=1,...,N.
    \label{cond:KDD}
\end{equation}

We define the matrix $\buK$ with entries of  indices $(i, j)$ 
given by  $K_0 (x_i - x_j)$ and our problem is reduced to finding a positive vector $P$ such that 
\[
\buK P = \mathds{1}, \quad \rho_i = \nu P_i \rho^2 \quad \text{and} \quad \rho = 1 / (\nu \mathds{1}^T P).
\]
Even though the matrix $\buK$ is symmetric with positive coefficients and constant main diagonal equal to $K_0 (0)$, and one can find invertibility conditions,  
%
it remains unclear whether $P := \buK^{-1} \mathds{1} > 0$ or not.

For this reason, an alternative viewpoint using a Lyapunov functional helps describing the asymptotically stable solutions, as detailed below.


\subsection{A Lyapunov concentration result: proof of Theorem \ref{thm:lyapunov}}
\label{ch3:subs:Lyapunov}

We consider the special form of equation \eqref{eq:resgeneralform} 
given by \eqref{eq:gennLyapunov}. Then, we consider an initial data $n^0 \in \calM_+ (\R)$ and define the probability measures 
\[ 
q(t,x) :=\frac{ n_{\eps} (\eps t, x)}{\rho_{\eps}(\eps t)}, 
 \quad \quad q^0 := \frac{n^0 }{\rho^0},
\]
There is a closed form equation for $q$, namely
\begin{equation}
\left\{
\begin{aligned}
 \p_t q(t,x) = \,& q(t,x) \big( \int_{\R} K_S(x,y) q(t,y) dy - R_0(x) \big) 
 \\
 &- q(t,x) \int_{\R} q(t,x') \big( \int_{\R} K_S(x',y) q(t,y) dy - R_0 (x') \big) dx', 
 \\
 q(0, x) = \, &q^0(x).
 \label{eq:genqLyapunov}
\end{aligned}
\right.
\end{equation}
Therefore we simply need to study the asymptotic behavior of $q$ as $t \to +\infty$ to be able 
to describe that of $n_{\eps}$ as $\eps \to 0$.

Notice that Equation \eqref{eq:genqLyapunov} has a replicator-type structure, as it can be written under the form
\begin{equation*}
\p_t q(t,x) =  q(t,x) \left[\mathcal{G}(x, q(t,\cdot)) - \int_\R q(t,x) \mathcal{G}(x, q(t,\cdot)) dx \right], 
\end{equation*}
where $\mathcal{G}(x, q(t,\cdot)) := \int_{\R} K_S(x,y) q(t,y) dy$ is the effective growth rate, or fitness, of 
$q(t,x)$ and the integral term $\int_\R q(t,x)\mathcal{G}(x, q(t,\cdot)) dx$ the average of this growth rate
in the population. This structure means that the frequency of the $x$-carrying individuals in the population evolves with
the deviation of its corresponding fitness from the mean fitness in the population.

\bigskip

Thanks to the structure of \eqref{eq:genqLyapunov}, we obtain the asymptotic stability of~$\delta_{x_M}$ 
stated in Theorem~\ref{thm:lyapunov}, by using the Lyapunov method for stability.
The rest of this section is devoted to the proof of this result, 
which relies on a strict Lyapunov stability argument
stated in \cite[Theorem 4.4]{Sandholm2001} and~\cite{Sandholm2010}.
This type of convergence result has appeared in the economic 
literature devoted to game theory with continuous strategy 
space, which we denote here by $\calP$ (for ``phenotype''). 
For instance it is stated in \cite[Theorem 3.a]{Cheung2016} 
and follows from \cite[Theorem 2]{Cheung2014}.

First, we get a Lyapunov functional for \eqref{eq:genqLyapunov} by defining
\begin{equation}
 J(q) := \f{1}{2} \iint_{\R^2} K_S(x,y)q(x)q(y) dx dy - \int_{\R} R_0 (x) q(x) dx.
 \label{eq:Lyapunov}
\end{equation}
Indeed, along an orbit of \eqref{eq:genqLyapunov} we have, thanks to \eqref{assumption:Ksym},
\begin{align*}
 \ddt J(q(t,\cdot)) = &\int_{\R} q(t,x) \big( \int_{\R} K_S(x,y) q(t,y) dy - R_0 (x) \big)^2 dx 
 \\
 &- \Big( \int_{\R} q(t,x) \big( \int_{\R} K_S(x,y) q(t, y) dy - R_0 (x) \big) dx \Big)^2 \geq 0,
\end{align*}
with equality (by the Cauchy-Schwarz inequality) if and only if 
\[
 \int_{\R} K_S(x, y) q(t, y) dy - R_0 (x) \equiv C \in \R \text{ on } \supp(q(t,\cdot)),
\]
so that we have strict monotonicity except if $q(t, \cdot)$ is a rest point for the dynamics of \eqref{eq:genqLyapunov}.
This Lyapunov functional can be seen as an embodiment of the ``positive correlation'' property from game dynamics (see \cite{HofSig2003}), and this feature has been exploited to get a gradient flow formulation of a non-local model with a diffusion term in \cite{JabLiu2017}, where the kernel acts for death induced by competition rather than for birth as is the case here.

In order to use the Lyapunov functional properly, we need $\{ q(t, \cdot), \, t \geq 0 \} \subset \calM_+^1 (\R)$ to be relatively compact for a topology for which $J$ is continuous and Fréchet-differentiable.
This is the case if either $q^0$ is compactly supported in $\R$, or 
if $K$ is bounded and $R_0$ is proper (by Prokhorov's 
theorem), for the weak-$\star$ topology on $\calM_+^1 (\R)$, hence the use of \eqref{assumption:R0}

The study of the maximizer sets for $J$ is greatly simplified by \eqref{assumption:Kpos}:
\begin{lemma}
 Under \eqref{assumption:Kpos}, the functional $J$ is strictly convex on the convex set $\calM_+^1(\R)$.
 
 Therefore its local maximum points are extreme points of $\calM_+^1 (\R)$, that is Dirac masses. The Dirac mass $\delta_x$ is a local maximizer of $J$ only if $y \mapsto K_S(x,y) - R_0 (y)$ reaches its maximum at $x$.
 \label{lem:convex}
\end{lemma}
\begin{proof}
 For $q_1, q_2 \in \calM_+^1(\R)$ and $\theta \in [0, 1]$ we compute
 \begin{align*}
  J(\theta q_1 + (1 -\theta) q_2) &= \iint_{\R^2} K_S \big( \theta^2 q_1 q_1 + (1 - \theta)^2 q_2 q_2 +  2 \theta (1 - \theta) q_1 q_2 \big) - \int_{\R} R_0 \big( \theta q_1 + (1 - \theta) q_2 \big)
  \\
  &= \theta J(q_1) + (1 - \theta) J(q_2) - \theta (1 - \theta) \iint_{\R^2} K_S(x,y) (q_1 -q_2)(x) (q_1 - q_2) (y) dx dy.
 \end{align*}
 Therefore \eqref{assumption:Kpos} (with $\xi = q_1 - q_2$) implies that $J$ is strictly convex.
 
 If $J$ reaches a local maximum at $\xi \in \calM_+^1 (\R)$ belonging to some interval $(\xi_-, \xi_+)$, that is $\xi = \theta \xi_- + (1 - \theta) \xi_+$ for some $\theta \in (0, 1)$ with $\xi_{\pm} \in \calM_+^1 (\R)$, then for $\eps > 0$ small enough we have
 \[
  J(\xi) < \f{1}{2} \big( J(\xi + \eps (\xi_+ - \xi_-)) + J(\xi - \eps (\xi_+ - \xi_-))\big) \leq J(\xi),
 \]
 where the left inequality holds by strict convexity and the right one by the local maximum condition. This is absurd, hence local maxima are only reached at extreme points.
 
 The support of an extreme probability measure must be reduced to a singleton: otherwise, we can construct a segment on which the measure lies by exchanging mass between any two separable points of the support. Conversely, a Dirac mass is obviously extreme, as any segment to which it belongs would consist of probability measures with the same support, reduced to a singleton.

 Then, the first-order optimality condition for $J$ at $\delta_x$ reads: for all admissible perturbation $h$,
 \[
  \int_{\R} \big( K(x,y) - R_0 (y) \big) h(y) dy \leq 0,
 \]
 and admissible perturbations have the general form $h = - \delta_x + h_0$, with $h_0 \in \calM_+^1 (\R)$, whence the last point.

\end{proof}

Thanks to \eqref{assumption:xM} we get that $\{\delta_{x_M} \}$ is a local 
maximizer set of~$J$ for which~$J$ is a strict Lyapunov function 
(and that there is no other local maximizer set of~$J$). Then, it follows
that $\{\delta_{x_M} \}$ is asymptotically stable.

\section{Models with mutations}\label{section:mutations}

To take into account mutations lead to much more elaborate tools that extend the methodology set in the previous section.
Our main results are $BV$ estimates stated in Propositions \ref{th_asfecundity} and \ref{th3}, which derivations can be understood in view of the simpler case in Section~\ref{sec:BVsimple}.
 
We begin with  $BV$ estimates on $\rho_\eps$ assuming that $R(x,\rho)=\nu \rho$ and then we address the difficulties encountered when $R$ has a general form and finally the Lyapunov method. .

\subsection{{\em BV} estimate for the AF model. Proof of Proposition  \ref{th_asfecundity}}
\label{ch3:estimBV}

Although the asymptotic behavior of $n_\eps$ solution to~\eqref{eq:resgeneralform} may be difficult to obtain in general, under some assumptions on $K$ and $R$, the total population $\rho_\eps$ can be proved to have bounded variations.

\medskip

Recall that, integrating equation \eqref{eq:resgeneralform}, we have
\[
\eps \dot{\rho_\eps} = \frac{1}{\rho_\eps} \iiint_{\R^3} K_{\eps}(x, y, z) n_\eps(t, y) 
n_\eps(t, z) dx \,dy \,dz - \int_\R R(x, \rho_\eps) n_\eps(t, x) \,dx.
\]
The proofs of Propositions \ref{th_asfecundity} and \ref{th3} rely on estimates 
obtained through the equation satisfied by~$\ddot{\rho}_\eps$. In general, we 
start from
\begin{equation}
\begin{aligned}\label{eq:ddrho}
\eps \ddot{\rho}_\eps = &- \frac{\dot{\rho}_\eps}{\rho_\eps^2} \iiint_{\R^3} K_{\eps}(x, y, z) 
n_\eps(t, y) n_\eps(t, z) dx \,dy \,dz \\
&+ \frac{1}{\rho_\eps}\iiint_{\R^3} K_{\eps}(x, y, z) \big( 
\p_t n_\eps(t, y) n_\eps(t, z) + n_\eps(t, y) \p_t n_\eps(t, z) \big) dx \,dy \,dz \\ 
&-\dot{\rho}_\eps \int_\R \p_{\rho} R (x, \rho) n_\eps(t, x) dx 
\\
&- \frac{1}{\eps}\int_\R R(x, \rho_\eps) \Big(\frac{1}{ \rho_\eps} \iint_{\R^2} K_{\eps}(x, y, z) n_\eps(t,y) 
n_\eps(t, z) dy \,dz - R(x, \rho_\eps) n_\eps(t, x) \Big) dx.
\end{aligned}
\end{equation}

\begin{proof}[Proof of Proposition \ref{th_asfecundity}]
We treat the case of the model with asymmetric fecundity. 
Then, $\rho_\eps$ satisfies 
\begin{equation*}
\eps \dot{\rho}_\eps = \int_{\R} B(x)  n_\eps(t, x) dx - \nu \rho_\eps^2,
\end{equation*}
and \eqref{eq:ddrho} reads
\begin{align*}
\eps \ddot{\rho} = &\int B(x) \p_t n_\eps(t, x) dx - 2 \nu \rho_\eps \dot{\rho}_\eps
\\
= &- \nu \rho_\eps \dot{\rho}_\eps + \frac{\nu^2}{\eps} \rho_\eps^3 - \frac{\nu \rho_\eps}{\eps} \int B(x) n_\eps(t, x) dx 
\\
&+ \frac{1}{\eps \rho_\eps} \iiint \alpha_{\eps}(x, y, z) B(x) B(y) n_\eps(t, y) n_\eps(t, z) dx \,dy \,dz - \frac{\nu \rho_\eps}{\eps} \int B(x) n_\eps(t, x) dx.
\end{align*}
We rewrite the last equation as
\begin{multline}
\eps \frac{d}{dt}\dot{\rho}_\eps =
 - \nu \rho_\eps \dot{\rho}_\eps + \overbrace{\frac{\rho_\eps}{\eps} \left( \frac{\int B(x) n_\eps(t, x) dx}{\rho_\eps} - \nu \rho_\eps \right)^2}^{\text{ demographic stabilization }}
 \\
 + \underbrace{\frac{1}{\eps \rho_\eps} \Big( \iiint \alpha_{\eps}(x, y, z) B(x) B(y) n_\eps(t, y) n_\eps(t, z) dx \,dy \,dz - \big( \int B(x) n_\eps(t, x) dx \big)^2 \Big)}_{\text{ mixing-induced fecundity variation }}.
\label{eq:ddrhoC2}
\end{multline}

In order to apply the same technique as for the simple case~\eqref{nomutation} in Section~\ref{sec:BVsimple}, 
we need to assume that the mixing-induced fecundity variation 
term is bounded from below. Under assumption~\eqref{assum:alphae}, we obtain from~\eqref{eq:ddrhoC2} and Proposition~\ref{rho:upperbound}
\begin{equation}
\eps \frac{d}{dt} \dot{\rho}_{\eps} \geq - \nu \rho_{\eps} \dot{\rho}_{\eps} - C.
\label{eq:ddrhoC2bis}
\end{equation}
From Proposition~\ref{prop:nonext3}, we deduce
\[
\frac{d}{dt} ( \dot{\rho}_{\eps} )_- \leq - \frac{\nu\rho_m}{\eps} (\dot{\rho}_{\eps})_- + \frac{C}{\eps},
\]
and thus 
\begin{equation*}
    (\dot{\rho}_{\eps})_- (t) \leq e^{-\frac{\nu\rho_m t}{\eps}}  (\dot{\rho}_{\eps})_- (0) + \frac{C}{\nu\rho_m} \big(1 - e^{-\frac{\nu\rho_m t}{\eps}} \big). 
\end{equation*}
Then we use the same argument as in the case without mutations, which proves uniform boundedness of $\rho_{\eps}$ in $BV(0, T)$ for all $T > 0$.
\end{proof}

\subsection{ {\em BV} estimate for the ATH model. Proof of Proposition \ref{th3}}

We now address the model with asymmetric trait 
heredity~\eqref{astrait}.

 Then we compute
\begin{align*}
\Big\lvert \int_\R \psi(x) ( G_{\eps} \ast \phi ) (x) dx - \int_\R \psi(x) 
\phi(x) dx \Big\rvert 
&\leq \int_\R \lvert \psi(x) \rvert \lvert G_{\eps} \ast \phi 
(x) - \phi(x) \rvert dx \\
&\leq \iint_{\R^2} \frac{1}{(2 \pi \eps^2)^{1/2}} e^{-\frac{(x-y)^2}{2 \eps^2}} \lvert 
\phi(y) - \phi(x) \rvert dy\, dx.
\end{align*}
We apply the change of variables $\hat{y} =  \eps^{-1} (y - x)$, then $d\hat{y} = \eps^{-1} dy$, and we obtain
\begin{align*}
\Big\lvert \int_\R \psi(x) ( G_{\eps} \ast \phi ) (x) dx - \int_\R \psi(x) 
\phi(x) dx \Big\rvert& \leq (2\pi)^{-1/2} \iint_{\R^2} e^{-\hat{y}^2/2} \lvert \phi(x +  \eps \hat{y}) - \phi(x) \lvert d\hat{y} \,dx \\
&\leq \frac{2 \lVert \phi' \rVert_{L^1}}{(2 \pi)^{1/2}} \eps.
\end{align*}

\medskip

\begin{proof}[Proof of Proposition \ref{th3}]

Departing from \eqref{astrait}, the equation satisfied by $\rho_\eps$ reads
\begin{equation*}
\eps \ddt \rho_\eps (t)= \int_\R \left(\frac{1}{\rho_\eps(t)}K \ast n_\eps(t,\cdot) (x)G_\eps\ast n_\eps(t,\cdot)(x)- \nu \rho_\eps(t) n_\eps(t,x)\right)dx.
\end{equation*}
Differentiating this equation, we obtain
\begin{align*}
\eps \ddot{\rho_\eps}(t)=\, &\frac{1}{\rho_\eps(t)}\int_\R \Big[K_0 \ast 
\p_t n_\eps(t,\cdot) (x)G_\eps\ast n_\eps(t,\cdot)(x) 
+ K_0 \ast  n_\eps(t,\cdot) (x)G_\eps\ast \p_t n_\eps(t,\cdot)(x)\Big]dx\\
&-\nu \rho_\eps(t)\dot{\rho}_{\eps}(t) -  \nu \int_\R \p_t n_\eps(t,x)\rho_\eps(t)dx\\
&- \frac{\dot{\rho}_\eps(t)}{\rho_\eps^2(t)} \int_\R \left[ K_0 \ast n_\eps(t,\cdot) (x)
G_\eps\ast n_\eps(t,\cdot)(x)\right]dx.
\end{align*}
By the same trick as in Section \ref{sub:simplecase}, assuming \eqref{assum:K0} induces
\begin{equation*}
\begin{aligned}
\eps \ddot{\rho_\eps}(t) =\,&\frac{1}{2\rho_\eps(t)}\ddt \left[\int_\R K_0 \ast n_\eps(t,\cdot) (x)G_\eps\ast n_\eps(t,\cdot)(x)dx\right]\\
&+\frac{1}{\rho_\eps(t)}\int_\R \left[G_\eps\ast(K_0 \ast n_\eps(t,\cdot)) (x) \p_t n_\eps(t,x) \right]dx\\
&- \nu \rho_\eps(t)\dot{\rho}_{\eps}(t)- \nu \int_\R \p_t n_\eps(t,x)\rho_\eps(t)dx\\
&-\frac{\dot{\rho}_\eps(t)}{\rho_\eps^2(t)} \int_\R \left[ K_0 \ast n_\eps(t,\cdot) (x)G_\eps\ast n_\eps(t,\cdot)(x)\right]dx.
\end{aligned}
\end{equation*}
Then we compute
\begin{align*}
\eps \ddot{\rho_\eps}(t) =&- \nu \rho_\eps(t)\dot{\rho}_{\eps}(t) + \frac{1}{2\rho_\eps(t)}\ddt \left[\int_\R K_0 \ast 
n_\eps(t,\cdot) (x)G_\eps\ast n_\eps(t,\cdot)(x)dx\right]\\
&+\frac{1}{\eps \rho_\eps(t)}\int_\R G_\eps\ast(K_0 \ast n_\eps(t,\cdot)) 
(x)\left[\frac{1}{\rho_\eps(t)}K_0 \ast n_\eps(t,\cdot) (x)G_\eps \ast 
n_\eps(t,\cdot)(x)- \nu n_\eps(t,x)\rho_\eps(t) \right]dx\\
&-\frac{\nu}{\eps}\rho_\eps(t) \int_\R \left[ 
\frac{1}{\rho_\eps(t)}K_0 \ast n_\eps(t,\cdot) (x)G_\eps\ast n_\eps(t,\cdot)(x)- \nu 
n_\eps(t,x)\rho_\eps(t)\right] dx\\
&-\frac{\dot{\rho}_\eps(t)}{\rho_\eps^2(t)} \int_\R \left[ K_0 \ast n_\eps(t,\cdot) 
(x)G_\eps\ast n_\eps(t,\cdot)(x)\right]dx,
\end{align*}
and get
\begin{align*}
\eps \ddot{\rho_\eps}(t) = &-\nu \rho_\eps(t)\dot{\rho}_{\eps}(t)+\frac{1}{2}\ddt 
\left[\int_\R \frac{1}{\rho_\eps(t)}K_0 \ast n_\eps(t,\cdot) (x)G_\eps\ast 
n_\eps(t,\cdot)(x)dx\right]\\
&-\frac{1}{2}\frac{\dot{\rho}_\eps(t)}{\rho_\eps^2(t)} \int_\R \left[ K_0 \ast 
n_\eps(t,\cdot) (x)G_\eps\ast n_\eps(t,\cdot)(x)\right]dx\\
&+\frac{1}{\eps} \int_\R( G_\eps \ast n_\eps) \left[G_\eps \ast (K_0 \ast n_\eps)\frac{(K_0 
\ast n_\eps)}{\rho_\eps^2}-2 \nu K_0 \ast n \right]dx \\
&+ \frac{1}{\eps} \nu^2 \rho_\eps^2 
\int_\R G_\eps \ast n_\eps \,dx.
\end{align*}
We rewrite this as
\begin{align*}
\eps \ddot{\rho_\eps} (t) = &- \nu \rho_{\eps}(t) \dot{\rho}_{\eps}(t) 
- \frac{1}{2} 
\frac{\dot{\rho}_{\eps} (t)}{\rho_{\eps}^2 (t)} \int_\R K_0 \ast n_{\eps} (t, 
\cdot)G_\eps\ast n_\eps(t,\cdot)\\
&+ \frac{1}{\eps}\int_\R (G_\eps \ast 
n_\eps)\left[\frac{(K_0 \ast n_\eps)}{\rho_\eps}-\nu \rho_\eps \right]^2 dx
\\
&+\frac{1}{\eps \rho_{\eps}^2 (t)} \int_\R (K_0 \ast n_{\eps} )(G_{\eps} \ast n_{\eps}) 
\Big( G_{\eps} \ast (K_0 \ast n_{\eps}) - K_0 \ast n_{\eps} \Big) dx.
\end{align*}

Now we use the convergence assumption \eqref{hyp:conv_ge} on $G_\eps$. 
We simply need to check that $\phi (x) := \int K_0 (x - y) n_{\eps} (t, y) dy$ 
is in $W^{1,1}$. This is obvious since $ \phi ' =  K'_0 \ast n_{\eps} $. 
Hence we have
\begin{equation}
\begin{aligned}
\label{rho_snd}
\frac{\eps}{2}\ddot{\rho_\eps}(t) =&-\nu \rho_{\eps} (t) \dot{\rho}_{\eps} (t) -\frac{1}{2}\frac{\dot{\rho}_\eps(t)}{\rho_\eps^2(t)} \int_\R K_0 \ast n_\eps(t,\cdot)G_\eps\ast n_\eps(t,\cdot)\\
&+ \frac{1}{\eps}\int_\R (G_\eps \ast n_\eps)\left[\frac{(K_0 \ast n_\eps)}{\rho_\eps}-\nu \rho_\eps \right]^2 dx + O(1).
\end{aligned}
\end{equation}

Thanks to \eqref{cond:nonextC4}, we deduce the inequality
\begin{equation*}
\frac{\eps}{2} \ddt (\dot{\rho}_\eps(t))_- \leq -\big(\frac{1}{2}\eta_0+ \nu \rho_{\eps} (t) \big) (\dot{\rho}_\eps(t))_- + O(1).
\end{equation*}
Then, we conclude that $\rho_\eps$ is bounded in $BV_{\text{loc}} (\R_+)$ uniformly in $\eps$. Indeed, we obtain that for some constants $C_1, C_2 > 0$,
\[
(\dot{\rho}_{\eps}(t))_- \leq e^{-C_1 t / \eps} \Big( (\dot{\rho}_{\eps} (0))_- + \frac{C_2}{\eps} \int_0^t e^{C_1 t'/\eps} dt' \Big),
\]
hence
\[
(\dot{\rho}_{\eps}(t))_- \leq (\dot{\rho}_{\eps} (0))_- e^{-C_1 t / \eps} + \frac{C_2}{C_1} \big( 1 - e^{-C_1 t / \eps} \big).
\]

As in the proof of Theorem \ref{th:simple}, we deduce that for all 
$T > 0$, $(\rho_{\eps})_{\eps}$ is uniformly in $\eps$ bounded 
in $BV([0, T])$ with assumption \eqref{assum:rho_0} on the initial data.
Going back to \eqref{rho_snd}, we derive the estimate, for $T>0$,
\begin{equation*}
\int_0^T\int_\R (G_\eps \ast n_\eps) \left[\frac{K_0 \ast n_\eps}{\rho_\eps}- \nu \rho_\eps\right]^2dx\,dt=O(\eps),
\end{equation*}
as in the proof of Theorem \ref{th:simple}.

\end{proof}

\subsection{Extensions and open questions for the general case}

As a  first possible extension, we address the case of a general saturation term for the AF model, 
featuring the competition effect and the trait-dependency:
\begin{equation}
R \in \mathcal{C}^1(\R^d \times \R_+; \R_+), \quad K(x,y,z) = B(y) \alpha_{\eps}(x, y, z), \quad \forall y, z, \, \int_\R \alpha_{\eps}(x,y,z) dx = 1.
\label{cond:C43}
\end{equation}

To apply the same argument as before, we need to assume
\begin{equation}
\begin{aligned}
&\exists C > 0, \, \forall \eps > 0, \, \forall y, z, \, \forall \phi \in L^1_+ \text{ with } \lVert \phi \rVert_{L^1} = 1,
\\
&\big\lVert \iint_{\R^2} \alpha_{\eps} (\cdot,y,z) B(y) \phi(y) \phi(z) dy dz - B(\cdot) \phi(\cdot) \big\rVert_{L^1} \leq C \eps,
\label{assum:alphaC4}
\end{aligned}
\end{equation}
and also
\begin{equation}
\forall \rho \leq \rho_M, C_f (\rho) := \lVert B(\cdot) - R(\cdot, \rho) \rVert_{\infty} < \infty, \quad \overline{C_f}=\sup_{0\leq\rho\leq\rho_M} C_f (\rho).
\label{assum:boundnetfitness}
\end{equation}

We are going to establish the following $BV$ estimate
\begin{corollary} \label{cor:BV_AF}
Assume \eqref{cond:C43}, \eqref{assum:alphaC4} and \eqref{assum:boundnetfitness}. Then, for all $T>0$, $(\rho_{\eps})_{\eps}$ is uniformly in $\eps$ bounded in~$BV([0, T])$.
\end{corollary}

In the case at hand, $\rho_\eps$ satisfies
\begin{equation*}
    \eps \dot{\rho_\eps}= \int_\R \left(B(x)- R(x, \rho_\eps)\right)n_\eps(t,x) dx.
\end{equation*} 
Differentiating this equation, we find
\begin{align*}
\eps \frac{d}{dt} \dot{\rho}_\eps = &\int_\R \big( B(x) - R(x, \rho_\eps) \big) \p_t n_\eps 
(t, x) dx - \dot{\rho}_\eps \int_\R \p_{\rho} R(x, \rho_\eps) n_\eps (t, x) dx \\
= &- \dot{\rho}_\eps \int_\R \p_{\rho} R(x, \rho_\eps) n_\eps (t, x) dx
+ \frac{1}{\eps} \int_\R n_\eps(t, x) \big(  B(x) - R (x, \rho_\eps) \big)^2 
\\
&+ \frac{1}{\eps} \int_\R \big( B(x) - R(x, \rho_\eps) \big) 
\Big(\frac{1}{\rho_\eps}\iint_{\R^2} \alpha_{\eps}(x, y, z) B(y) n_\eps(t,y) n_\eps(t,z) dy \,dz 
- B(x) n_\eps(t,x) \Big) dx .
\end{align*}
The last term can be seen as the integral of the net fitness $B-R(\cdot, 
\rho_\eps)$ weighted by a fecundity deviation $\Delta_{n_\eps(t, \cdot)} B$,
with $\int \Delta_{n_\eps(t, \cdot)} B (x) dx = 0$.

Under assumptions \eqref{assum:alphaC4} and \eqref{assum:boundnetfitness}, this additional term is treated as in the case $R(x,\rho)=\nu \rho$, replacing the negative constant on the right-hand side of~\eqref{eq:ddrhoC2bis} by $- \rho_M C \overline{C_f}$, which gives
\begin{equation*}
\eps \ddt \dot{\rho_\eps}(t) \geq - \dot{\rho}_\eps \int_{\R} \p_{\rho} R(x, \rho_\eps) n_\eps (t, x) dx -\rho_M C \overline{C_f}.
\end{equation*}
and, following  the proof the proof of Proposition~\ref{th_asfecundity}, we obtain Corollary~\ref{cor:BV_AF}.

\medskip

A second possible extension is a general death term for the ATH model:
\begin{equation}
R \in \mathcal{C}^1 (\R^d \times \R_+; \R_+), \quad K_\eps(x, y, z) = G_{\eps} (x - z) K_0(x-y).
\label{cond:C43bis}
\end{equation}
In order to see clearly where the difficulty lies, we replace $G_{\eps} (x - z)$ by $\delta_{x = z}$ (letting $\eps \to 0$ in this term only) and prove the
\begin{corollary} \label{cor:bv2}
If $R(x, \rho) = R_1 (\rho)$ and $\rho R'_1(\rho) \geq R_1(\rho)$, then $\dot{\rho_\eps} \leq 0$ implies $\ddot{\rho_\eps} \geq - \frac{\dot{\rho_\eps}}{2 \rho_\eps} \int n \zeta$. Then in particular for all $T > 0$, $(\rho_{\eps})_{\eps}$ is uniformly in $\eps$ bounded in $BV([0, T])$.
\end{corollary}
Notice that, for instance, the assumption on $R_1$ holds for $R_1(\rho) = \nu \rho^{\gamma}$ for some $\gamma \geq 1$. 
\\

 For simplicity, we define 
\[
\zeta(t, x) := \frac{K_0 \ast n_\eps(t, \cdot)(x)}{\rho_\eps(t)}, \quad Q(t) := \int_{\R} \p_{\rho} R(x, \rho_\eps(t)) n_\eps(t, x) dx.
\]
After computations similar to the previous ones, we find
\begin{equation}
\frac{1}{2} \eps \frac{d}{dt} \dot{\rho}_\eps = - \frac{\dot{\rho}_\eps}{2 \rho_\eps} \int_{\R} n_\eps \zeta + \frac{1}{\eps} \int_{\R} n_\eps \Big[ \zeta^2 - R \zeta + \frac{R+Q}{2}\big( R - \zeta \big) \Big],
\label{eq:ddrhoC5}
\end{equation}
and the term in $\frac{1}{\eps}$ rewrites
\[
\int_{\R} n (\zeta - \frac{R+Q}{2}) (\zeta - R).
\]
Meanwhile, one can check that
\[
\eps \dot{\rho_\eps} = \int_{\R} n_\eps (\zeta - R).
\]

When $\dot{\rho}_\eps \leq 0$ we would like to prove that the term in $\frac{1}{\eps}$ in \eqref{eq:ddrhoC5} is non-negative. We could be less restrictive and simply require $\ddot{\rho}_\eps \geq 0$. This reads (with $q_\eps(t, x) = n_\eps(t, x) / \rho_\eps (t)$):
\[
\int_{\R} q_\eps(t, x) \big( \zeta (t, x) - R(x, \rho_\eps(t)) \big) \Big( \zeta (t, x) - \frac{R(x, \rho_\eps(t)) + Q(t)}{2} - \int_{\R} q_\eps(t, y) \zeta (t, y) dy \Big) dx \geq 0
\]
if
\[
\int_{\R} q_\eps(t, x) \big( \zeta(t, x) - R(x, \rho_\eps(t)) \big) dx \leq 0.
\]

A straightforward computation gives the Corollary~\ref{cor:bv2} but other (more general) cases can be treated similarly.


\subsection[Discussion on the Lyapunov approach]{Discussion on the Lyapunov approach \sectionmark{Lyapunov approach}}
\sectionmark{Discussion on the Lyapunov approach}
\label{ch3:Lyapunov}

We may also discuss the Lyapunov approach applied to reproduction terms including mutations. As in Section \ref{ch3:subs:Lyapunov}, we define 
$q_{\eps}(t,x) := n_{\eps} (\eps t, x)/\rho_{\eps}(\eps t)$. 
For equation~\eqref{eq:resgeneralform}, 
assuming $R(x,\rho) = R_0 (x) + R_1 (\rho)$, we find
\begin{equation}
\left\{ 
\begin{aligned} \p_t q_{\eps} (t,x) =  &\iint_{\R^2} K_{\eps} (x, y, z) q_{\eps} (t, y) q_{\eps} (t,z) dy dz - R_0 (x) q_{\eps} (t,x)
\\
 &- q_{\eps}(t,x)\Big( \iiint_{\R^3} K_{\eps} (x',y,z) q_{\eps}(t,y) q_{\eps} (t, z) dx' dy dz - \int_{\R} R_0 (x') q_{\eps} (t, x') dx' \Big), 
\\
q_{\eps}(0, x) =\,  &q^0_{\eps} (x).
\end{aligned}
\right.
\label{eq:moregenqLyapunov}
\end{equation}

A natural candidate Lyapunov functional is given by
\[
 J^{\eps} (q) := \f{1}{2} \iiint_{\R^3} K_{\eps}^S (x,y,z) q(y) q(z) dx dy dz - \int_{\R} R_0 (x) q(x) dx,
\]
\[
 K_{\eps}^S (x, y, z) := \f{K_{\eps} (x, z, y) + K_{\eps} (x, y, z)}{2}.
\]
Then, we can compute along an orbit of~\eqref{eq:moregenqLyapunov}
\begin{align*}
 \ddt J^{\eps} (q_{\eps}(t, \cdot)) &= \int_{\R} \p_t q_{\eps} (t, y) \big( \iint_{\R^2} K_{\eps}^S (x, y, z) q_{\eps} (t, z) dz dx - R_0 (y) \big) dy
 \\
 &= \idotsint_{\R^5} K_{\eps}^S (x, y, z) K_{\eps}^S (y, y', z') q_{\eps} (t, y') q_{\eps} (t, z') q_{\eps} (t, z) dz' dy' dz dy dx
 \\
&- \iiint_{\R^3} K_{\eps}^S(x,y,z) \big( R_0 (x) + R_0 (y) \big) q_{\eps} (t, y) q_{\eps} (t, z) dx dy dz + \int_{\R} q_{\eps} (t, x) R_0^2 (x) dx 
\\
&- \Big( \iiint_{\R^3} K_{\eps}^S (x,y,z) q_{\eps} (t, y) q_{\eps} (t, z) dx dy dz - \int_{\R} q_{\eps} (t, y) R_0 (y) dy \Big)^2.
\end{align*}

In the special case $R_0 \equiv 0$, to get a non-decreasing $J^{\eps}$ along orbits, we need to assume
\begin{multline}
\forall \xi \in \calM_+^1 (\R), \quad \idotsint_{\R^5} K_{\eps}^S (x, y, z) K_{\eps}^S (y, y', z') \xi(y') \xi(z') \xi(z) dz' dy' dz dy dx
\\ \geq \Big( \iiint_{\R^3} K_{\eps}^S (x, y, z) \xi(y) \xi(z) dx dy dz \Big)^2,
\label{hyp:lyap1}
\end{multline}
which could be interpreted as an increase of fecundity from parents to offspring, with equality only if the dynamic is at rest, that is
\[
 \iint_{\R^2} K_{\eps}^S (\cdot, y, z) \xi(y) \xi(z)\, dydz \text{ is constant on } \supp(\xi).
\]
In other words, to obtain a Lyapunov functional requires a perfect analogue of the Cauchy-Schwarz inequality.

When $R_0 \neq 0$, this Lyapunov functional also applies for~\eqref{asfecundity} with constant $B$, that is under the assumption
\[
 \exists B > 0,  \, \forall y, z, \quad \int_{\R} K_{\eps} (x, y, z) dx = B.
\]
Then,  we write $K_{\eps} = B \alpha_{\eps}$ and get $J^{\eps} (q) = \f{B}{2} - \int_{\R} q(y) R_0 (y) dy$ so that
\begin{multline*}
 \ddt J^{\eps} (q_{\eps}(t, \cdot)) = \int_{\R} q_{\eps}(t,y) R_0^2 (y) dy - \big(\int_{\R} q_{\eps}(t,y) R_0 (y) dy \big)^2 
 \\
 + B \big( \int_{\R} q_{\eps}(t,y) R_0 (y) dy - \iiint_{\R^3} R_0 (x) \alpha_{\eps} (x, y, z) q_{\eps} (t, y) q_{\eps}(t, z) dx dy dz \big).
\end{multline*}
To get that  $J^{\eps}$ is non-decreasing along orbits, one possible additional assumption is therefore
\begin{equation}
 \forall \xi \in \calM_+^1 (\R), \quad \int_{\R} R_0 (y) \xi(y) dy \geq \iiint_{\R^3} R_0 (x) \alpha_{\eps}(x,y,z) \xi(y)\xi(z) dx dy dz,
 \label{hyp:lyap2}
\end{equation}
which could be interpreted as a decrease of the death rate from parents to offspring.
\\

These two conditions could be combined for more generality.  However, more realistic assumptions such as \eqref{assum:alphae}, \eqref{assum:alphaC4} or \eqref{hyp:conv_ge} do not imply that $J^{\eps}$ itself is a non-decreasing Lyapunov function, but rather that along an orbit of~\eqref{eq:moregenqLyapunov},
\[
 \ddt J^{\eps} (q_{\eps} (t,\cdot)) = j^0 (q_{\eps} (t, \cdot)) + \eps j^1_{\eps} (q_{\eps} (t, \cdot)),
\]
where $j^1_{\eps}$ is uniformly bounded, and $j^0(q) \geq 0$ with equality if and only if $q$ is a rest point of the limit dynamics.
In other words, we get Lyapunov stability {\itshape asymptotically} as $\eps \to 0$. 
The possible outcomes of this approach are still to be investigated.

\section{The Hamilton-Jacobi equation} \label{ch3:HJ}

In the context of evolutionary dynamics, the Hamilton-Jacobi 
approach has been introduced in~\cite{Dieketal2005} and then 
developed in~\cite{BarPer2008, LorMirPer2011} to study the 
concentration effect for phenotypically structured PDE models 
of asexual populations. This approach
consists in determining the possible Dirac distributions through the zeros of $u_\eps$ defined from the Hopf-Cole transform
\begin{equation*}
u_\eps(t,x)=\eps \ln n_\eps(t,x).
\end{equation*}
In the mentioned works, the convergence of $u_\eps$ as $\eps$ goes to 0 is rigorously established and the limit~$u$ satisfies a constrained Hamilton-Jacobi equation, using the theory of viscosity solutions (see~\cite{CIL1992, Barles1994} for an introduction). The constraint on the solution $u$ reads
\begin{equation*}
\max_{x \in \R} u(t,x)=0, \quad\forall t>0,
\end{equation*}
and comes from the control in $L^1$ of the total population. 
Then, some properties on the concentration points can be 
derived from the study of this constrained Hamilton-Jacobi 
equation and the solution~$u$. In some particular cases, 
it is proved that the population density remains monomorphic, 
that is composed of a single Dirac mass, and then a form 
of canonical equation is derived, giving the dynamics 
of the dominant trait.

\bigskip

In the present work, this Hamilton-Jacobi structure 
arises in the different situations that were previously studied.
We prove in this section Theorem~\ref{th:HJ}, which 
states different results on 
the regularity of $u_\eps$ and a constraint on the limit $u$. 
Then, we deal with the limiting Hamilton-Jacobi
equations and the consequences of Theorem~\ref{th:HJ} to discuss the potential
concentration points.

The statements of Theorem~\ref{th:HJ} concern, for both models \eqref{asfecundity} and \eqref{astrait},
the convergence of $u_\eps$ as~$\eps$~vanishes, up to extratction of 
subsequences, and the existence of a uniform
upper bound on $u_\eps$ that converges to 0. To prove the first point, 
we derive a priori estimates 
on $u_\eps$, and then on its derivatives, in order to use compactness arguments. 
The second point relies on these derived estimates.

The uniqueness of the solution to the limit equation has not been 
proved in our context (see~\cite{CalvezL2020} for the most general result so far), thus we only derive convergence up to extraction of 
subsequences. Moreover, the stability result is not complete : the convergence of $u_\eps$ to a solution of the 
limiting constrained Hamilton-Jacobi equation, at least
in the sense of viscosity, remains to be rigorously proved.
The main obstacles to the proof we encounter are the 
time-dependency of the coefficients and their lack of regularity.

In this section, we first derive the limiting Hamilton-Jacobi equations 
associated to some particular forms of \eqref{asfecundity} and \eqref{astrait},
and introduce the assumptions that are needed in the proof of 
Theorem~\ref{th:HJ}. The proof is deferred to Appendix~\ref{appendixA}. 
Then, we discuss the formal limits of $u_\eps$ and $n_\eps$, regarding 
the concentration of the population. Finally we present the consequences in the case
of the no mutation model \eqref{nomutation}, for which we can conclude the monomorphic 
behavior of the population density.

\subsection{Derivation of the constrained H-J equations}

\medskip

\noindent{\bfseries Asymmetric fecundity:} we use the particular form $\alpha_\eps (x,y,z)=\frac{1}{\eps} \alpha\left(\frac{x-z}{\eps}, y\right)$, that is 
\begin{equation*}
K_\eps(x, y, z) = B(y) \frac{1}{\eps}\alpha\left(\frac{x-z}{\eps}, y\right) \text{ with } \int_{\R} \alpha(z', y) dz' = 1 \text{ for all } y,
\end{equation*}
and we define
\begin{equation}\label{def:b_e}
r_\eps(t,x):=R(x, \rho_\eps (t)), \quad q_\eps(t,y)=  \frac{n_\eps(t,y)}{\rho_\eps (t)}.
\end{equation}
With these notations, and going back to \eqref{asfecundity}, the equation on $u_\eps$ reads 
\begin{equation}\label{eq:AF_HJe}
\p_t u_\eps(t,x) = \int_{\R} B(y) q_\eps(t,y) 
\int_{\R} \alpha(z,y) e^{\frac{u_\eps(t,x- \eps z)-u_\eps(t,x)}{\eps}}dz\,dy - r_\eps(t,x),
\end{equation}
and we compute the formal limiting equation
\begin{equation}\label{AF_HJ}
\begin{aligned}
\p_t u(t,x) &= \int_{\R} B(y) q(t,y)  \int_{\R} \alpha(z,y) e^{-\p_x u(t,x)\cdot z}dz\,dy - r(t,x)\\
&=\int_{\R} B(y) q(t,y)  \mathcal{L}[\alpha(\cdot,y)](\p_x u(t,x))dy - r(t,x),
\end{aligned}
\end{equation}
with $\mathcal{L}[\alpha(\cdot,y)]$ the Laplace transform of $\alpha(\cdot,y)$ for all $y$:
\begin{equation*}\label{laplace}
\mathcal{L}[\alpha](p):=\int_{\R} \alpha(z) e^{-p \cdot z}dz,
\end{equation*}
for $\alpha$ a probability density function.

\bigskip

\noindent{\bfseries Asymmetric trait heredity:} The interest of this problem comes from the time- and trait-dependent coefficients of the Hamiltonian. We use the generic form
\begin{equation*}\label{eq:generic_ath}
K_\eps(x,y,z)=G_\eps(x-z)K_1(x,y).
\end{equation*}
Going back to \eqref{astrait}, and after the change of variable $z'=\frac{x-z}{\eps}$, the equation on $u_\eps$ reads
\begin{equation}\label{eq:ATH_HJe}
\p_t u_\eps(t,x) = \frac{1}{\rho_\eps (t)}\int_{\R} K_1(x,y) n_\eps(t,y) dy \cdot \int_{\R} G(z') e^{\frac{u_\eps(t,x- \eps z')-u_\eps(t,x)}{\eps}}dz' - r_\eps(t,x).
\end{equation}
For clarity, we define
\begin{equation}\label{def:a_e}
k_\eps(t,x):=\int_{\R} K_1(x,y)q_\eps(t,y)dy.
\end{equation}
At the limit $\eps \to 0$, we obtain the formal limiting equation
\begin{equation}\label{ATH_HJ}
\begin{aligned}
\p_t u(t,x) &=  k(t,x) \int_{\R} G(z) e^{-\p_x u(t,x) \cdot z}dz - r(t,x)\\
&=k(t,x) \,\mathcal{L}[G](\p_x u(t,x))-r(t,x),
\end{aligned}
\end{equation}
with $a$ and $b$ the formal limits of $a_\eps$ and $b_\eps$ defined in \eqref{def:a_e} and \eqref{def:b_e}, and $\mathcal{L}[G]$ the Laplace transform of $G$. From now on, we choose $G$ such that its Laplace transform is well defined on $\R$.

In the case $G$ is the gaussian density, the equation on $u_\eps$ reads
\begin{equation}\label{eq:gauss_ath}
\p_t u_\eps(t,x) = k_\eps(t,x) \int_{\R} \frac{1}{\sqrt{2\pi}}e^{-\frac{\vert z \vert^2}{2}} e^{\frac{u_\eps(t,x- \eps z)-u_\eps(t,x)}{\eps}}dz - r_\eps(t,x).
\end{equation}
Then, passing formally to the limit $\eps \to 0$, we arrive at
\begin{equation*}
\begin{aligned}
\p_t u(t,x) &=  k(t,x) \int_{\R} \frac{1}{\sqrt{2\pi}}e^{-\frac{\vert z \vert^2}{2}} e^{-\p_x u(t,x) \cdot z}dz - r(t,x)
\\[1mm]
&=k(t,x) \,e^{\frac{(\p_x u(t,x))^2 }{2}}-r(t,x).
\end{aligned}
\end{equation*}

The complete proof of Theorem \ref{th:HJ} is deferred to Appendix~\ref{appendixA}, since it uses 
quite standard and technical arguments.
We mostly focus on Equation~\eqref{eq:gauss_ath}, but the methods are identical for the generic ATH case. 
The proof of the theorem in the AF case is similar and we also give the formal ideas where it is necessary.

\vspace{3mm}

\noindent{\bfseries Assumptions for Theorem \ref{th:HJ}:} We assume on the function $R$
\begin{equation}
\label{Hyp:croissance_R}
\exists C_0 >0, \, \forall \rho_m \leq \rho \leq \rho_M, \, \forall x \in \R, \quad R(x, \rho) \leq C_0(1+|x|),
\end{equation}
\begin{equation}\label{Hyp:R}
\exists L_r >0, \forall \rho_m \leq \rho \leq \rho_M, \, \forall x \in \R, \quad \vert \p_x R(x,\rho) \vert \leq L_r. 
\end{equation}
We choose the positive function $K_1$ bounded
\begin{equation}\label{bound:K1}
\exists \bar K>0, \forall x,y \in \R, \quad K_1(x,y) \leq \bar K,
\end{equation}
and such that,
\begin{equation}\label{Hyp:aeps}
\exists \lambda>0, \exists C_\lambda >0,\forall \eps>0, t\geq0, x\in \R, \quad e^{\frac{|\p_x k_\eps(t,x)|}{\lambda k_\eps(t,x)}} \lambda k_\eps(t,x) \leq C_\lambda.
\end{equation}
This assumption is satisfied for example when $K_1$ is bounded and there exists a constant $L_K$ such that
\begin{equation*}
|\p_x K_1(x,y)| \leq L_K |K_1(x,y)|, \quad \forall x, y \in \R,
\end{equation*}
or, when $K_1$ induces a gaussian type distribution for $a_\eps$, that is,
\begin{equation*}
k_\eps (t,x) \sim C e^{\frac{-(x-m)^2}{\sigma^2}}.
\end{equation*}
We also assume on the initial condition
\begin{equation}
\label{initial_u}
u_\eps^0 (x) \leq -A |x| + C, \quad \|\p_x u^0_\eps \| \leq L_0.
\end{equation}
For the model with asymmetric fecundity, we assume that $B$ 
and $\alpha$ are positive and bounded. 

\subsection{Limiting Hamilton-Jacobi equations}
\label{sec:polymorphism}

In the context of viscosity solutions, cf.~\cite{CIL1992, Barles1994}, 
the use of the stability property enables to prove the convergence 
of $u_\eps$ to a solution to the corresponding constrained Hamilton-Jacobi equation,
from which we can deduce some information on the potential concentration points.
Despite the lack of regularity of the considered Hamiltonians, we make 
here some comments on the limiting equations we obtained.

As it is classically proved with the Hamilton-Jacobi approach to adaptive dynamics, the limit function $u$ satisfies the constraint
\begin{equation}
\max_{x \in \R} u(t,x)=0, \quad\forall t>0,
\end{equation}
because of the control on the total population density.
Then, when $u$ is differentiable at maximum points, we deduce that $\p_t u$ and $\p_x u$ are equal to 0 and, going back to \eqref{AF_HJ} and \eqref{ATH_HJ}, 
we obtain, for $\bar n$ the formal limit of $n_\eps$,
\begin{equation}
\text{supp } \bar n \subset \{ (t,x)\in (0,\infty)\times \R | b(t) - r(t,x)=0 \}, \quad \text{in case \eqref{asfecundity}},
\end{equation}
\begin{equation}
\text{supp } \bar n \subset \{ (t,x)\in (0,\infty)\times \R | k(t,x) - r(t,x)=0 \}, \quad \text{in case \eqref{astrait}},
\end{equation}
where $b(t)$ and $k(t,x)$ are the limits 
\begin{equation}
    b(t) = \int_\R B(y) q(t,y) dy, \quad k(t,x) = \int_\R K_1(x,y) q(t,y) dy.
\end{equation}

It would be then interesting to determine the conditions required to have these null sets reduced to an isolated point. If, for all $t>0$, we identify a unique point $\bar x (t)$ satisfying
\[
b(t)-r(t, \bar x (t))=b (t)-R(\bar x (t), \bar \rho (t) )=0, \quad \text{in case \eqref{asfecundity}},
\]
\[
k(t, \bar x (t))-R(\bar x (t), \bar \rho (t) )=0, \quad \text{in case \eqref{astrait}},
\]
then the population is monomorphic, that composed of a single Dirac mass located on $\bar x (t)$.

\bigskip

Apart from providing a description of the behavior of $u_\eps$, and then of $n_\eps$,
as $\eps$ vanishes, the constrained Hamilton-Jacobi equation usually enables to identify 
the set of points where the population would concentrate. Thence, we can derive under specific regularity 
assumptions a canonical equation, which is a differential equation 
giving the dynamics of the dominant trait in a monomorphic population. But in the cases considered in the present work, 
because of the form of the reproduction terms, the Hamiltonians feature integral terms of the measure $q_\eps$,
whose limits are not explicit as $\eps$ goes to $0$.
Thus, the identification of monomorphic or polymorphic limit is a difficult question.

\bigskip

However, we show a particular case where we deduce a monomorphic state from the study of the population 
at equilibrium. More precisely, we prove that the population cannot be composed of several Dirac masses.

We go back to \eqref{nomutation} and define $\overline{n} \in \calM_+ (\R)$ as an Evolutionary Stable Distribution (ESD) in the sense of \cite{Desetal2008,JabRao2011}, that is
 \begin{align}
  K_0 \ast \overline{n} &= \nu \overline{\rho}^2 \text{ on } \supp(\overline{n}),
  \label{cond:ESD1}
  \\
  K_0 \ast \overline{n} &\leq \nu \overline{\rho}^2 \text{ on } \R,
  \label{cond:ESD2}
 \end{align}
 where $\overline{\rho} = \int \overline{n}$.
The interest of the ESD concept is huge: it is readily established that a stationary solution to \eqref{nomutation} is asymptotically stable if and only if it satisfies \eqref{cond:ESD1} and \eqref{cond:ESD2}.

If we assume that $K_0$ is radial-decreasing, then we prove that 
extreme points in $\supp (\overline{n})$ (if it is bounded) cannot 
support a positive Dirac mass, by using \eqref{cond:ESD2}. In particular, among 
all combinations of Dirac masses, only the single-point measure 
$\overline{n}_{\overline{x}} (x) := K_0 (0) / \nu \delta_{x = \overline{x}}$ is 
an ESD.

Indeed, assume that $\bar n$ is composed of $k \geq 2$ Dirac masses located on $(x_i)_{1\leq i \leq k}$, then defining
\[
\overline{K}(x) := K_0 \ast \bar n (x) =\sum_{i=1}^k \rho_i K_0(x-x_i).
\]
Then, we deduce from \eqref{cond:ESD1} and \eqref{cond:ESD2} that $\overline{K}$ is maximal on the support of $\bar n$, that is the points $x_i$. With no loss of generality, we assume that the sequence $(x_i)$ is ordered and $x_1=\min_{i} x_i $. Then, differentiating $\overline{K}$, we obtain
\[
\overline{K}'(x_1) =\sum_{i\geq 1} \rho_i K_0'(x_1-x_i) >0,
\]
which contradicts the optimality of $\overline{K}$ on the support of $\bar n$.
Hence the population, at the asymptotic limit, cannot be polymorphic.

\section{Conclusion and perspectives}

We investigated adaptive dynamics for population dynamics model including 
sexual reproduction, when the trait is mainly inherited from 
the mother. We determined non-extinction conditions and a 
control on the total population. In the particular case of a 
saturation term $R$ depending only on the competition, we 
derived BV estimates on the total population. In general, 
estimating the variations of $\rho_\eps$ when $R$ depends 
on both trait variable and competition seems difficult, 
and  a Lyapunov functional approach yields 
complementary results under some structure conditions. An open problem is to find another method allowing for more appropriate 
assumptions in order to get stability results.

Concerning the sequences $u_\eps=\eps \ln n_\eps$ associated to each model, 
we obtained local Lipschitz estimates uniform in $\eps$. 
To deduce the convergence of $u_\eps$ to the solution of 
the limiting Hamilton-Jacobi equation with constraint, 
we still need time compactness on the coefficients of 
\eqref{eq:AF_HJe} and~\eqref{eq:ATH_HJe}. As a special 
case of both, for the Hamilton-Jacobi equation associated 
to the model without mutations \eqref{eq:gennLyapunov}, 
if we provide some convergence result on 
$\int K(x,y) \ast n_\eps(t,y)/\rho_{\eps}(t)$ and 
on $\rho_{\eps}$, then, up to extraction of a subsequence, 
the limit function $u$ has an explicit formulation and 
its maximum points can be described. In general, 
Hamilton-Jacobi equations with time- and space-dependent 
coefficients are difficult to deal with when there is a 
lack of regularity. The authors in~\cite{LioSou2000} 
developed a theory of stochastic viscosity solutions 
to tackle nonlinear stochastic PDEs. In particular, 
they prove existence, regularity and uniqueness results 
for the viscosity solution when the time-dependent 
coefficient of the Hamiltonian can be written as the 
derivative of a trajectory. This theory does not apply 
to our models since the coefficients in front of the 
gradient-dependent term are not under the form of a 
time derivative.

Another question is the determination of a convenient 
framework to observe Dirac concentrations. The convergence 
of the population distribution to a sum of Dirac masses 
illustrates the selection of well-adapted or dominant 
phenotypical traits. In \cite{LorMirPer2011, CJ2011}, the 
Hamilton-Jacobi approach enables to characterize the 
dynamics of the dominant traits under specific assumptions 
of regularity. In our framework, the required hypotheses to prove Dirac 
concentrations are to be clarified.


Using the Wasserstein distance has been recently developed in 
\cite{MagRao2015, DFR, FP2021} to derive asymptotics of population distributions for similar equations. It is proved that specific cases of the  
sexual reproduction operator, possibly in an infinitesimal model, 
induce a control, possibly a contraction, for the Wasserstein distance on 
the phenotypical trait space. It could 
be interesting to further explore this method in full generality.

\section*{Acknowledgements}
The  authors  are  very  thankful  to  Pierre-Alexandre Bliman
for proposing the biological motivations and the directions that led to this work. 
B.P. has received funding from the European Research 
Council (ERC) under the European Union's Horizon 2020 research and innovation program (grant 
agreement No 740623).

\begin{appendix}

\section{Proof of Theorem \ref{th:HJ}}\label{appendixA}

\subsection{A priori bounds}

We begin with the estimates for the ATH case, and especially with a gaussian trait female heredity distribution.
\begin{lemma}
    Let $u_\eps$ be solution to equation \eqref{eq:AF_HJe} or \eqref{eq:gauss_ath}. Then, there exist constants $C_1>0$ and $C_2>0$, such that for all $t>0, x\in \R$ and $\eps>0$ we have
    \begin{equation*}
     -C_1(1+t)(1+ |x|)  \leq u_\eps (t,x) \leq -A |x| + C_2(1+t).
    \end{equation*}
\end{lemma}
We prove this lemma in the case of a gaussian trait female heredity distribution, but the argument exactly applies to equation~\eqref{eq:ATH_HJe} in the generic ATH case.

\begin{proof}
We first prove the lower bound
\begin{equation*}
\label{apriori_up}
u_\eps(t,x) \geq -C_1(1+t)(1+ |x|).
\end{equation*}
Indeed, because $k_\eps \geq 0$ and $\mathcal{L}(G)\geq 0$, we deduce from \eqref{Hyp:croissance_R} that
\[
\p_t u_\eps \geq - r_\eps (t,x) \geq -C_0(1+|x|).
\]
From \eqref{initial_u} we obtain
\begin{equation*}
u_\eps (t,x) \geq \inf_\eps u_\eps^0(0) - \inf_\eps \|\p_x u^0_\eps \| -C_0 t(1+|x|).
\end{equation*}
Hence the lower bound.

We also derive the inequality
\begin{equation*}
u_\eps (t,x) \leq -A |x| + C_2(1+t),
\end{equation*}
where $C_2=\bar K \frac{1}{\sqrt{2\pi}}\int e^{-\vert z \vert^2/2} e^{A\vert z \vert}dz$. 
Indeed, defining $v(t,x):=-A |x| + C_2(1+t)$, we compute
\[
\p_t v(t,x)-k_\eps(t,x)  \int \frac{1}{\sqrt{2\pi}}e^{-\vert z \vert^2/2} e^{\frac{v(t,x- \eps z)-v(t,x)}{\eps}}dz \geq C_2-\bar K \frac{1}{\sqrt{2\pi}}\int e^{-\vert z \vert^2/2} e^{A\vert z \vert}dz \geq 0.
\]
Thus, $v$ is a super-solution of \eqref{eq:gauss_ath}, and since $u^0(x)\leq v(0,x)$ we deduce that $u_\eps(t,x)\leq v(t,x)$ by a comparison principle argument. 
\end{proof}
We obtain the same kind of bounds for the asymmetric fecundity case, with the constant 
$C_2:= \sup_{y}B(y)\int \alpha(z,y) e^{|A|z}dz$..

\subsection{Regularity in space}
We prove the following
\begin{lemma}
Let $u_\eps$ be the solution to the equation \eqref{eq:gauss_ath}. For $\lambda>0$ given by \eqref{Hyp:aeps} and for all $t>0,x \in \R$, we have
\begin{equation*} 
|\p_x u_\eps(t,x)| \leq \| \p_x u_\eps^0 \|_{L^{\infty}} + (C_\lambda +L_r)t +\lambda  \left( \sup_{\eps} \| u^0_\eps \|_{L^{\infty}} + C_1(1+t)(1+ |x|) \right).
\end{equation*}
This implies that $u_\eps$ is Lipschitz in space, uniformly in $\eps$ and locally in time.
\end{lemma}
\begin{proof}
We use the notations
\begin{equation*}
p_\eps(t,x)=\p_x u_\eps (t,x), \quad p(t,x)=\p_x u(t,x).
\end{equation*}
Differentiating \eqref{eq:gauss_ath}, $p_\eps$ satisfies
\begin{align*}
\p_t p_\eps (t,x) &= \p_x k_\eps(t,x) \cdot \int \frac{1}{\sqrt{\pi}}e^{-\vert z \vert^2} e^{\frac{u_\eps(t,x- \eps z)-u_\eps(t,x)}{\eps}}dz \\
&+  k_\eps(t,x) \int \frac{1}{\sqrt{\pi}}e^{-\vert z \vert^2} e^{\frac{u_\eps(t,x- \eps z)-u_\eps(t,x)}{\eps}}\left(\frac{p_\eps(t,x- \eps z)-p_\eps(t,x)}{\eps}\right)dz - \p_x r_\eps(t,x).
\end{align*}
Let $\lambda>0$. We define 
\begin{equation*}\label{def:w_D}
w_\eps^\lambda (t,x)=p_\eps(t,x) + \lambda u_\eps (t,x),\quad D_\eps (t,x,z)=\frac{u_\eps(t,x- \eps z)-u_\eps(t,x)}{\eps}. 
\end{equation*}
Then, $w^\lambda_\eps$ satisfies
\begin{align*}
\p_t w^\lambda_\eps &=k_\eps \cdot \int \frac{1}{\sqrt{\pi}}e^{-\vert z \vert^2} e^{D_\eps(t,x,z)}\left(\frac{w^\lambda_\eps(t,x- \eps z)-w^\lambda_\eps(t,x)}{\eps}\right)dz\\
&- \lambda \left[ k_\eps \cdot \int \frac{1}{\sqrt{\pi}}e^{-\vert z \vert^2} e^{D_\eps(t,x,z)}(D_\eps(t,x,z) - 1)\right]dy\\
&+\p_x k_\eps \cdot \int \frac{1}{\sqrt{\pi}}e^{-\vert z \vert^2} e^{D_\eps(t,x,z)}dz - (\p_x r_\eps + \lambda r_\eps).
\end{align*}
Then, using \eqref{Hyp:R}, we have
\begin{align*}
\p_t w^\lambda_\eps &-L_r - k_\eps \cdot \int \frac{1}{\sqrt{\pi}}e^{-\vert z \vert^2} e^{D_\eps}\left(\frac{w^\lambda_\eps(t,x- \eps z)-w^\lambda_\eps(t,x)}{\eps}\right)dz \\
&\leq \int \frac{1}{\sqrt{\pi}}e^{-\vert z \vert^2}e^{D_\eps}\left[ \p_x k_\eps+ \lambda k_\eps -\lambda k_\eps D_\eps \right]dz.
\end{align*}
Defining $f(D):=e^D (\p_x k_\eps+\lambda k_\eps- \lambda k_\eps D)$, the maximum of $f$  on $\R$ is reached at $D^*:=\frac{\p_x k_\eps}{\lambda k_\eps}$ and equals
\begin{equation*}
e^{\frac{\p_x k_\eps}{\lambda k_\eps}} \lambda k_\eps \leq C_\lambda,
\end{equation*}
from \eqref{Hyp:aeps}. Then we have the upper bound
\begin{equation*}
w^\lambda_\eps(t,x) \leq \max_\R w_\eps^\lambda (0,x) + Ct, \quad C= C_\lambda+L_r,
\end{equation*}
which implies the upper bound on $p_\eps$
\begin{equation*}
p_\eps(t,x) \leq \| \p_x u_\eps^0 \|_{L^{\infty}} + Ct +\lambda  \left( \sup_{\eps} \| u^0_\eps \|_{L^{\infty}} + C_1(1+t)(1+ |x|) \right).
\end{equation*}
We have the same estimate for $-p_\eps$.
\end{proof}

For the AF model, we have the following estimate on the derivative in space of $u_\eps$:
\begin{lemma}
Let $u_\eps$ be the solution of equation \eqref{eq:AF_HJe}. Then, for all $t>0,x \in \R$ and $\eps>0$, we have
\begin{equation*} 
|\p_x u_\eps(t,x)| \leq \| \p_x u_\eps^0 \|_{L^{\infty}} + L_r t.
\end{equation*}
This implies that $u_\eps$ is Lipschitz in space, uniformly in $\eps$ and locally in time.
\end{lemma}

We address the limit equation
\begin{equation}\label{eq:afHJ}
\p_t p(t,x)=(-\p_x p(t,x))\int B(y) q(t,y)  \int z \alpha(z,y) e^{-p(t,x)\cdot z}dzdy - \p_x r(t,x),
\end{equation}
and give formal arguments, since the proof for the $\eps$-level problem is similar to the one of the ATH case.
We compute that $w(t):=\| \p_x u_\eps^0 \|_{L^{\infty}} + L_r t$ is a super-solution of \eqref{eq:afHJ}. Since $p(0,x)\leq w(0)$ for all $x\in \R$, we deduce that, from the comparison principle, $u_\eps$ is Lipschitz in space, uniformly in $\eps$ and locally in time.
\subsection{Regularity in time}
In the ATH case, since we proved that $u_\eps$ is uniformly Lipschitz in space locally in time, we can deduce that $\p_t u_\eps$ is locally uniformly bounded.

\begin{lemma}
Let $u_\eps$ be the solution to equation \eqref{eq:ATH_HJe} and let $T>0$ and $\bar r>0$ be fixed. Assume \eqref{Hyp:R} and \eqref{bound:K1}. Then, there exists $C(T,\bar r)>0$ such that, for all $t \in [0,T], x \in B(0,\bar r)$, we have
\begin{equation*}
\vert \p_t u_\eps \vert\leq C(T,\bar r)+ \sup_{0 \leq \rho \leq\rho_M}\| R(\cdot, \rho) \|_{L^{\infty}(B(0,\bar r))}.
\end{equation*}
This implies that $u_\eps$ is Lipschitz in time, uniformly in $\eps$.
\end{lemma}

\begin{proof}
Let $T>0$ and $\bar R>\bar r>0$ be fixed with $\bar R$ large enough. We choose some constants $L_1$ and $L_2$ such that
\[
u_\eps(t,x) < - L_1, \quad \forall (t,x) \in [0,T] \times \R \backslash B(0,\bar R),
\]
\[
|p_\eps|< L_2, \quad \forall (t,x) \in [0,T] \times B(0,\bar R).
\]
Then, we obtain for $t \in [0,T], x \in B(0,\bar r)$,
\begin{multline*}
\vert \p_t u_\eps \vert \leq \sup_{0 \leq \rho \leq\rho_M}\| R(\cdot, \rho) 
\|_{L^{\infty}(B(0,\bar r))} \phantom{\int}\\
+\frac{1}{\rho_\eps (t)}\int K(x,z) n_\eps(t,z) dz \cdot 
\left(\int_{|x-\eps y|<\bar R} \!\!\!\!e^{-\vert y \vert^2} e^{L_2 y}dy + \int_{|x-\eps 
y|>\bar R} \!\!\!\!e^{-\vert y \vert^2} e^{\frac{u_\eps(t,x- \eps y)-u_\eps(t,x)}{\eps}}dy  
\right).
\end{multline*}
Thus, for $\eps$ small enough, and assuming that 
\[
u_\eps(t,x) > - L_1, \quad \forall t \in [0,T], \forall x \in B(0,\bar r),
\]
\[
u_\eps(t,x) < - L_1, \quad \forall t \in [0,T], \forall x \in \R \backslash B(0,\bar R),
\]
we have
\begin{align*}
\vert \p_t u_\eps \vert &\leq \overline{K} \left(\int_{|x-\eps y|<\bar R} 
\!\!\!\!\!\!\!\!\!\!e^{-\vert y \vert^2} e^{L_2 y}dy + \int_{|x-\eps y|>\bar R} 
\!\!\!\!\!\!\!\!\!\!e^{-\vert y \vert^2} e^{\frac{-L_1-u_\eps(t,x)}{\eps}}dy  
\right)+ \sup_{0 \leq \rho \leq\rho_M}\| R(\cdot, \rho) \|_{L^{\infty}(B(0,\bar r))} 
\\
&\leq \overline{K} \left(\int e^{-\vert y \vert^2} e^{L_2 y}dy + \int_{|x-\eps y|> \bar R} e^{-\vert y \vert^2} dy\right)+ \sup_{0 \leq \rho \leq\rho_M}\| R(\cdot, \rho) \|_{L^{\infty}(B(0,\bar r))} \\
&\leq \overline{K} \left(\int e^{-\vert y \vert^2} e^{L_2 y}dy + \sqrt{\pi}\right)+ \sup_{0 \leq \rho \leq\rho_M}\| R(\cdot, \rho) \|_{L^{\infty}(B(0,\bar r))}.
\end{align*}
Hence the local uniform bound on $\p_t u_\eps$. 
\end{proof}

The proof is similar for the AF case.
\begin{lemma}
Let $u_\eps$ be the solution to equation \eqref{eq:AF_HJe} and let $T>0$ and $r>0$ be fixed. Then, there exists $C(T,r)>0$ such that, for all $t \in [0,T], x \in B(0,r)$, we have
\begin{equation*}
\vert \p_t u_\eps \vert\leq C(T,\bar r)+ \sup_{0 \leq \rho \leq\rho_M}\| R(\cdot, \rho) \|_{L^{\infty}(B(0,\bar r))}.
\end{equation*}
This implies that $u_\eps$ is Lipschitz in time, uniformly in $\eps$.
\end{lemma}

\subsection{A more precise upper bound}
The following argument concerns both cases and gives a sharper upper bound on~$u_\eps$.

\begin{lemma}
Let $u_\eps$ be the solution to equation \eqref{eq:AF_HJe} or \eqref{eq:ATH_HJe}. Then, for all $x,y \in \R$, we have
\begin{equation*}
u_\eps (t, x) \leq \eps \ln \big(\rho_M m_{x, \frac{C (1 + t)}{\eps}} \big),
\end{equation*}
where $m_{x, A} > 0$ is the minimum on $\R$ of $g_{x, A} : y \mapsto A \frac{1 + \max(\lvert x \rvert, \lvert y \rvert)}{1 - e^{-\lvert y-x \rvert A (1 + \max(\lvert x \rvert, \lvert y \rvert))}}$. 

In addition, if $A > 0$ we have $A < m_{x, A} \leq A + 3/2$. Thus, we obtain the global upper bound
\[
 u_\eps (t, x) \leq \eps \ln \big(  \rho_M (3/2 + C (1+t) / \eps) \big) \xrightarrow[]{\eps \to 0} 0.
\]
\end{lemma}

\begin{proof}
For all $z \in (x, y)$, by the mean value theorem there exists $\theta_{\eps} (t, x, z)$ between $x$ and $y$ such that
\[
    u_\eps (t, z) = u_\eps (t, x) + (z - x) \p_x u_\eps (t, \theta_\eps (t, x, z)).
\]
In addition, by the previous point there exists $C$ (independent of $t, x$ and $\eps$) such that for all $t, x$, $\lvert \p_x u_\eps (t, x) \rvert \leq C (1 + t) (1 + \lvert x \rvert)$.
Hence
\[
u_\eps (t,z) \geq u_\eps (t, x) - (z - x) C (1 + t) \big(1 + \max(\lvert x \rvert, \lvert y \rvert) \big).
\]
Since we have, for $x < y$,
\[
\int_x^y e^{\frac{u_\eps (t, z)}{\eps}} dz \leq \rho_M,
\]
we deduce that
\[
\eps e^{\frac{u_\eps(t,x)}{\eps}} \frac{1 - e^{-(y - x) \frac{C (1 + t) (1 + \max(\lvert x \rvert, \lvert y \rvert) )}{\eps}}}{C (1 + t) \big(1 + \max(\lvert x \rvert, \lvert y \rvert) \big)}\leq \rho_M, \quad \forall y.
\]
Then, we compute
\[
    u_\eps (t, x) \leq \eps \ln \Big( \frac{\rho_M C (1 + t) \big(1 + \max(\lvert x \rvert, \lvert y \rvert) \big)}{\eps \big(1 - e^{-(y-x) \frac{C (1 + t) (1 + \max(\lvert x \rvert, \lvert y \rvert) )}{\eps}} \big)} \Big),
\]
and this holds for all $y > x$.
We can also choose $y < x$ and get in more generality
\begin{equation*}
    u_\eps (t, x) \leq \eps \ln \Big( \frac{\rho_M C (1 + t) \big(1 + \max(\lvert x \rvert, \lvert y \rvert) \big)}{\eps \big(1 - e^{- \lvert y-x \rvert \frac{C (1 + t) (1 + \max(\lvert x \rvert, \lvert y \rvert) )}{\eps}} \big)} \Big)=\eps \ln \big(\rho_M g_{x, \frac{C (1 + t)}{\eps}} (y)\big).
\end{equation*}
Observe that $g_{x,A}$ is positive and goes to $+\infty$ at $y=\pm \infty$ and at $y =x$. Minimizing in $y$, we find that
\[
u_\eps (t, x) \leq \eps \ln \big(\rho_M m_{x, \frac{C (1 + t)}{\eps}} \big).
\]
To conclude we first remark that if $A > 0$ and $x, y \in \R$, then we have
\[
 \frac{1+\max(\lvert x \rvert, \lvert y \rvert)}{1-e^{-\lvert y - x \rvert A (1 + \max(\lvert x \rvert, \lvert y \rvert)}} > 1,
\]
so $g_{x, A} (y) > A$ for all $y \in \R$ and thus $m_{x, A} > A$. Then, with $A > 0$ we also have
\begin{align*}
 g_{1/A,A} (-1/A) = \frac{A+1}{1 -e^{-2 (1 + A)}} \leq A + 3/2,
\end{align*}
which implies $m_{x, A} \leq A + 3/2$. Thus, we obtain the global upper bound
\[
 u_\eps (t, x) \leq \eps \ln \big(  \rho_M (3/2 + C (1+t) / \eps) \big) \xrightarrow[]{\eps \to 0} 0.
 \]
\end{proof}
The proof of Theorem \ref{th:HJ} is achieved.

\end{appendix}

\bibliographystyle{abbrv}
\bibliography{biblio.bib}

\end{document}